\documentclass[twoside,10pt,a4paper,envcountsame]{amsart}

\usepackage[utf8]{inputenc}
\usepackage[german, english]{babel}
\usepackage[centertags]{amsmath}
\usepackage{amsmath}
\usepackage{amssymb}
\usepackage{amsfonts}
\usepackage{amsxtra}
\usepackage{bbm}
\usepackage{mathtools}
\usepackage{dsfont}
\usepackage{color}
\usepackage{hyperref}

\usepackage{tikz}
\usepackage{pgfplots}
\tikzset{slopetriangle/.style={
  bottom color=black!20,
  middle color=black!5,
  top color=white,
  draw=black
}}
\pgfplotsset{
  width=.65\linewidth,
  axis background/.style={fill=black!5!white},
  grid style={densely dotted,semithick},
  legend style={
    legend columns=1,
    legend pos=outer north east
  },
  compat=newest 
}
\pgfplotscreateplotcyclelist{MyColors}{%
    {black,mark = *,every mark/.append style={solid,scale=0.4,fill=black}},
    {red,mark = diamond*,every mark/.append style={solid,scale=0.4,fill=red}},
    {green,mark = square*,every mark/.append style={solid,scale=0.4,fill=green}},
    {blue,mark = halfcircle*,every mark/.append style={solid,scale=0.4}},   
    {yellow,mark = diamond*,every mark/.append style={solid,scale=0.4,fill=yellow}},
    {dashed,every mark/.append style={solid,scale=0.8,fill=black}},
    {dotted,every mark/.append style={solid,scale=0.8,fill=black}}
}

\usepackage[hyperref,comment,enumitem,theorem]{paper_diening}
\usepackage{algorithm2e}

\usepackage{enumitem}
\usepackage{ifthen}
\usepackage{graphicx}

\definecolor{lightgrey}{rgb}{.7,.7,.7}

\newcommand{\calJ}{\mathcal{J}}
\newcommand{\calG}{\mathcal{G}}

\newcommand{\calN}{\mathcal{N}}
\newcommand{\calO}{\mathcal{O}}

\newcommand{\Jeps}{\mathcal{J}_\epsilon}
\newcommand{\Jepsn}{\mathcal{J}_{\epsilon_n}}

\newcommand{\Aeps}{A_\epsilon}
\newcommand{\Veps}{V_\epsilon}

\newcommand{\phieps}{\phi_\epsilon}
\newcommand{\kappaeps}{\kappa_\epsilon}
\renewcommand{\div}{\operatorname{div}}
\renewcommand{\d}{\, d}
\renewcommand{\leq}{\leqslant}
\renewcommand{\le}{\leqslant}
\renewcommand{\geq}{\geqslant}
\renewcommand{\ge}{\geqslant}

\newcommand{\eps}{\varepsilon}
\newcommand{\constrained}[1]{\epslower\vee #1\wedge \epsupper}

\newcommand{\epslower}{\eps_-}
\newcommand{\epsupper}{\eps_+}

\DeclareMathOperator*{\Gammalim}{\Gamma\text{-}\lim}

\numberwithin{equation}{section}

\begin{document}

\title{A Relaxed Ka\v{c}anov  Iteration for the $p$-Poisson
  Problem}

\author{L.~Diening}
\address{Lars Diening, Bielefeld University,
Faculty of Mathematics,
 Postfach 10 01 31,
 D-33501 Bielefeld, Germany}
\email{lars.diening@uni-bielefeld.de}
\thanks{This research was supported by the DFG project
    ``Optimal Adaptive Numerical Methods for p-Poisson Elliptic
    Equations''.
  We acknowledge support for the Article Processing Charge by the Deutsche Forschungsgemeinschaft and the Open Access Publication Fund of Bielefeld University.}

\author{M.~Fornasier}
\address{Massimo Fornasier, TUM,
 Department of Mathematics,
 Boltzmannstraße 3, 85748 Garching (Munich), Germany}
\email{massimo.fornasier@ma.tum.de}

\author{R.~Tomasi}
\address{Roland Tomasi, Osnabr\"uck University, Germany}

\author{M.~Wank}
\address{Maximilian Wank, Osnabr\"uck University, Germany}

\maketitle

\begin{abstract}
  In this paper we introduce and analyze an iteratively re-weighted
  algorithm, that allows to approximate the weak
  solution of the $p$-Poisson problem for $1 < p \leq 2$ by
  iteratively solving a sequence of linear elliptic problems. The
  algorithm can be interpreted as a relaxed Ka{\v c}anov iteration, as
  so-called in the specific literature of the numerical solution of
  quasi-linear equations. The main contribution of the paper is
  proving that the algorithm converges at least with an algebraic
  rate.

\end{abstract}


\section{Introduction}

In this paper we approach the numerical solution of the $p$-Poisson
problem
\begin{align}
  \label{eq:pLap}
  \begin{aligned}
    -\operatorname{div}(|\nabla u|^{p-2}\nabla u)  &= f\quad\text{in }\Omega,\\
    u&= 0\quad\text{on }\partial\Omega,
  \end{aligned}
\end{align}
where $\Omega \subset \setR^d$ is open and bounded and $1<p<\infty$.
The solution might be scalar or vector-valued.\footnote{All our
  results also hold for the $p$-Poisson system, where the functions are
  vector-valued. To this end sometimes~$\setR^d$ has to be replaced
  by~$\setR^{N \times d}$.}

Nonlinear problems of this type appear in many applications,
e.g. non-Newtonian fluid theory~\cite{Lad67}, turbulent flow of a gas
in porous media, glaciology or plastic modeling. Moreover, the
$p$-Laplacian has a similar model character for nonlinear problems as
the ordinary Laplace operator for linear problems; see~\cite{Lin06}
for an introduction.

As usual we are looking for the weak solution of~\eqref{eq:pLap}. In
particular, we are searching for a function $u\in W_0^{1,p}(\Omega)$
such that
\begin{align}\label{eq:weakppoisson}
  \int\limits_\Omega |\nabla u|^{p-2} \nabla u \cdot\nabla \xi \d x =
  \langle f,\xi\rangle \quad\forall \xi\in W_0^{1,p}(\Omega),
\end{align}
where in the most general case $f\in (W_0^{1,p}(\Omega))^\ast$.  It is
well-known that the solution is unique and coincides with the
minimizer of the energy $\calJ: W_0^{1,p}(\Omega)\to\setR$ defined by
 \begin{align}
   \label{eq:J}
   \calJ(v) := \tfrac 1p\int\limits_\Omega |\nabla v|^p \d x -\langle
  f,v\rangle .
\end{align}
Due to the nonlinearity of the problem it is harder to obtain
efficient numerical solutions of this problem with a guaranteed
performance.  Our goal is to construct solutions
of~\eqref{eq:weakppoisson} by means of a numerically accessible
algorithm. In particular, we construct an iterative algorithm that
approximates solutions of~\eqref{eq:weakppoisson}, where in each step
only a linear elliptic problem has to be solved. Primarily, we focus
here on the iteration on the infinite dimensional
space~$W^{1,p}_0(\Omega)$. However, the same algorithm will
immediately apply also to discretized versions of the $p$-Poisson
problem, e.g., by means of finite elements or wavelets.  This approach
would coincide with the one adopted, for instance, in~\cite{CU05} of
first finding an iteration on the infinite-dimensional solution space
and then discretizing in space. We will consider in subsequent work
the effect of the discretization and its adaptation to error
estimators.

In this paper we also restrict ourselves to the case $p\in(1,2]$, since we
are in particular interested in relatively small values of~$p$, also because the case of $p>2$ is already addressed to a certain extent in~\cite{CU05}.  We will see, e.g., in
Example~\ref{ex:peak} that our algorithm actually only works properly for the
range of~$p \in (1,2]$.


Coming from the weak formulation \eqref{eq:weakppoisson} one can
interpret the problem as a weighted Poisson problem
\begin{align}
  \label{eq:div_a_nabla_u}
  \int\limits_\Omega a^{p-2} \nabla u\cdot\nabla \xi\d x =\langle
  f,\xi\rangle\quad\forall\xi\in W_0^{1,2}(\Omega)
\end{align}
for the given $f$, where $a:\Omega\to\setR$ and $a=|\nabla u|$. This
suggests to iteratively calculate for a given function $v_n$ the new
iterate $v_{n+1}$ as the solution of
\begin{align*}
  \int\limits_\Omega |\nabla v_n|^{p-2} \nabla v_{n+1}\cdot \nabla \xi
  \d x = \langle f,\xi\rangle\quad\forall\xi\in W_0^{1,2}(\Omega).
\end{align*}
The advantage of this step is that the calculation of~$v_{n+1}$ only
requires solving a linear problem. This allows invoking relatively standard appraoches
to discretize this step and solve it numerically with guaranteed performances. The problem with this approach, however, is
that the weighted Poisson problem is only well posed if $a$ is bounded
from above and from below away from zero. However, the weight $|\nabla
v_n|^{p-2}$  may be degenerating, at points where $\abs{\nabla u}=0$ or
$\abs{\nabla u}=\infty$. 

To overcome this problem we will use
a relaxation arguments. Therefore, we introduce in our algorithm two relaxation parameters
$\epsilon_-, \epsilon_+ \in (0,\infty)$ with $\epsilon_- \le
\epsilon_+$ that ensure that the weight is truncated properly from
below and above. In particular, we replace~$a$ by its truncation
\begin{align*}
  \epsilon_- \vee a \wedge \epsilon_+ := \max \set{ \epsilon_-, \min
    \set{ a, \epsilon_+}}.
\end{align*}
Note that this is just the (pointwise) closest point projection of~$a$
to the truncation interval~$ [\epsilon_-,\epsilon_+]$.  The
limits $\epsilon_- \searrow 0$ and $\epsilon_+ \nearrow \infty$ will
recover the unrelaxed or original problem. We also write
$\epsilon:=[\epsilon_-, \epsilon_+]$ and interpret~$\epsilon$ both as
a pair  $\{\epsilon_-,\epsilon_+\}$ and as the truncation
interval~$[\epsilon_-,\epsilon_+]$. We will
write~$\epsilon \to [0,\infty]$ as a short version
of~$\epsilon_- \searrow 0$ and~$\epsilon_+ \nearrow \infty$.  We will
see later, see Corollary~\ref{cor:regCM}, that for~$f$ in the Lorentz
space~$L^{d,1}(\Omega)$ the
lower parameter~$\epsilon_-$ is the crucial one.

According to these considerations we propose the following algorithm:

\begin{algorithm}[H]
\SetAlgoLined
\TitleOfAlgo{The relaxed $p$-Ka{\v c}anov{} algorithm}
\KwData{Given~$f \in (W^{1,p}_0(\Omega))^*$, $v_0 \in
  W^{1,2}_0(\Omega)$;}
\KwResult{Approximate solution of the $p$-Poisson
  problem~\eqref{eq:weakppoisson}; }
Initialize: $\eps_0=[\eps_{0,-},\eps_{0,+}] \subset (0,\infty)$,
$n=0$\;

\While{desired accuracy is not achieved yet}{
  Define $a_n := \eps_{n,-} \vee |\nabla v_n| \wedge
  \eps_{n,+}$\;

  Calculate $v_{n+1}$ by means of

  \begin{align*}
     \hspace*{-8mm}\int\limits_\Omega (\eps_{n,-} \vee |\nabla v_n| \wedge
     \eps_{n,+})^{p-2}\nabla v_{n+1}\cdot \nabla \xi \d x = \langle
     f,\xi\rangle \qquad \forall \xi \in W_0^{1,2}(\Omega);\;
  \end{align*}

  Choose new relaxation interval $\eps_{n+1} \supset \epsilon_n$;

  Increase $n$ by 1\;


}{
}
\end{algorithm}

Since $0 < \epsilon_{n,-} \leq \epsilon_{n,+}<\infty$ the equation
for~$v_{n+1}$ in the algorithm is always well defined, since it is
uniformly elliptic (with constant depending on~$\epsilon_n$).

This algorithm is not completely new in the realm of quasi-linear
equations. Such an iterative linearization approach is in fact called
the Ka{\v c}anov method in \cite{MR2754575,MR1454708} and we refer to
those papers for additional references related to the history of this
method for solving numerically quasi-linear equations.  It was also
proposed and analyzed to solve total variation minimization problems
in image processing, which can be formally related to the 1-Laplace
differential operator in \cite{MR1667392,MR1440119}.

Unfortunately, the results obtained in these aforementioned papers
cannot be applied straightforwardly to justify the convergence of the
Ka{\v c}anov iteration for equations involving the $p$-Laplace
operator.  In particular, to obtain quantitative estimates of
convergence with precise rates, as we do in this paper, one needs to
employ several finer tools, which have been explored in, e.g.,
\cite{DE08,DieK08,BelDieKre11,RD06}, precisely to handle singularities in
nonlinear differential operators such as the $p$-Laplacian. In
particular, the theory of N-functions, Orlicz spaces \cite{KR61},
shifted N-functions~\cite{DE08} and Lipschitz truncations, see
\cite{DKS13} and \cite{BDS16} have been used
systematically in the analysis of such nonlinear operators, allowing
the development of a potential theory analogous to the one known of
linear equations.

Besides these tools from nonlinear potential theory, the variational
formulation of the algorithm, as introduced first in \cite{MR1440119},
and further used to analyze other related iteratively re-weighted
least squares algorithms \cite{MR2588385,MR2869510}, offers the right
framework for the analysis also of the $p$-Ka{\v c}anov iteration.

Taking inspiration from \cite{MR1440119,MR2588385}, in
Section~\ref{sec:algo} we provide the variational derivation of this
algorithm based on the alternating minimization of a relaxed energy
with two parameters.

If we apply the algorithm with fixed relaxation parameter~$\epsilon$
independent on~$n$, i.e. $0<\epsilon_- \leq \epsilon_+ < \infty$, then
our iterates~$v_n$ converge to the unique minimizer~$u_\epsilon$ of
another  one-parameter relaxed energy~$\mathcal{J}_\epsilon$. We study this limit in
Section~\ref{sec:nlimit} and present (linear) exponential rates of convergence.

In Section~\ref{sec:epslimit} we study how the minimizers~$u_\epsilon$
of the relaxed energy~$\Jeps$ converge to the minimizer~$u$ of the
original problem.  This convergence can also be interpreted as a limit
in the sense of $\Gamma$-convergence
\cite{MR1968440,MR1201152}. Differently, e.g., from \cite{MR1440119},
we use a novel argument based on the Lipschitz truncation technique to
establish a recovery sequence for the $\Gamma-\lim\sup$. In
particular, thanks to the finer tools mentioned above, we can go beyond
a pure compactness argument as provided by the $\Gamma$-limit and
derive precise rates of convergence depending on~$\epsilon$.

Finally, in Section~\ref{sec:algebraic} we combine the estimates of
the two previous sections to deduce an overall error analysis with
algebraic rates.

\section{Variational Formulation of the Algorithm}
\label{sec:algo}

In this section we show that the algorithm can be deduced from an alternating minimization of a
relaxed energy. Recall that $1 < p \leq 2$ throughout this
article. Since the case $p=2$ is just the standard Laplace problem, it
suffices in the following to consider the case~$1<p<2$ only.

Let us introduce some standard notation. We use~$W^{1,p}(\Omega)$ and
$W^{1,p}_0(\Omega)$ for the Sobolev space without and with zero
boundary values. We use~$c$ for a generic positive constant whose value may
change from line to line. We  use~$f \lesssim g$ for $f \leq c\,
g$. We also write $f \eqsim g$ for $f \lesssim g$ and $g \lesssim f$.

The most important feature of the algorithm is that it only needs to
solve linear sub-problems, which carry their own energy depending on the weight. Therefore, very much inspired by the work \cite{MR1440119,MR2588385} and with appropriate adjustments, we extend the energy  by an additional parameter $a \,:\, \Omega \to [0,\infty)$ such that
the new functional is quadratic with respect to~$v$. In particular, we
define
\begin{align*}
  \mathcal{J}(v,a) := \int\limits_\Omega \tfrac 12 a^{p-2}|\nabla
  v|^2 + (\tfrac 1p - \tfrac 12)a^p\d x - \langle f,v\rangle .
\end{align*}
This energy is well-defined for all~$v \in W^{1,p}_0(\Omega)$ and
measurable $a\,:\, \Omega \to [0,\infty)$ but might take the
value~$\infty$.

This relaxed energy is convex with respect to~$(v,a)$. This follows
from the fact that $\beta(t,a) := \frac 12 a^{p-2}t^2$ is convex on
$[0,\infty)^2$, since
\begin{align*}
  (\nabla^2 \beta)(t,a) = \begin{pmatrix} a^{p-2} & (p-2)a^{p-3}t
    \\
    (p-2)a^{p-3}t & \frac 12 (p-2)(p-3)a^{p-4}t^2
  \end{pmatrix}
\end{align*}
is nonnegative definite as $a^{p-2}\geq 0$ and
$\det ((\nabla^2 \beta)(t,a))=a^{2p-6} t^2 (2-p)(p-1)\geq 0$. Notice
that in the latter lower bound we specifically used $1 < p \leq 2$.

\begin{remark}
  \label{rem:psmaller2}
  If $p>2$, then the relaxed energy~$\mathcal{J}(v,a)$ is neither
  bounded from below nor convex with respect to~$a$. Therefore, the
  algorithm derived below using the minimization with respect to~$a$ does not lead to a feasible problem for~$p>2$. See also Remark~\ref{rem:peak3}.
\end{remark}

Note that $\mathcal{J}(v,a)$ (for fixed~$a$) is quadratic with respect to~$v$
and a minimization with respect to~$v$ leads formally to the elliptic equation
\begin{align*}
  -\divergence(a^{p-2} \nabla v) &= f,
\end{align*}
see~\eqref{eq:div_a_nabla_u} for its weak form.

Unfortunately, the ellipticity of this system degenerates for
$a(x)\to 0$ and $a(x) \to\infty$. To overcome this problem we restrict
the minimization with respect to~$a$ (for fixed~$v$) to functions with
values within a relaxation
interval~$[\epsilon_-, \epsilon_+] \subset (0,\infty)$,
i.e. $\epsilon_- \leq a(x) \leq \epsilon_+$.  This minimization with
respect to~$a$ (for fixed~$v$) has a simple solution, namely
\begin{align}
  \label{eq:argmina_res}
  \argmin_{a\,:\,\epsilon_- \leq a \leq \epsilon_+} \mathcal{J}(v,a) =
  \epsilon_- \vee \abs{\nabla v} \wedge \epsilon_+,
\end{align}
where $\vee$ denotes the maximum and $\wedge$ the minimum, since
\begin{align*}
  \tfrac{\partial}{\partial a} \Big(  \tfrac 12 a^{p-2}|\nabla
  v|^2 + (\tfrac 1p - \tfrac 12)a^p \Big) = \tfrac{2-p}{2}
  a^{p-3} (a^2 - \abs{\nabla v}^2).
\end{align*}
This allows us to define for fixed~$\epsilon= [\eps_-,\eps_+] \subset
[0,\infty]$ another relaxed energy
\begin{align}
  \label{eq:def_Jeps}
  \Jeps(v) := \calJ\left(v,\constrained{\vert\nabla v\vert}\right) =
  \min_{a\,:\,\epsilon_- \leq a \leq \epsilon_+} \mathcal{J}(v,a).
\end{align}
This immediately implies that the relaxed
energy~$\mathcal{J}_\epsilon(v)$ is monotonically decreasing with
respect to~$\epsilon$, i.e., an increasing interval~$\epsilon$ in terms of inclusion
decreases the energy~$\mathcal{J}_\epsilon(v)$.

This new relaxed energy~$\mathcal{J}_\epsilon$ somehow ``hides'' the
constrained minimization with respect to~$a$.  We can
write~$\mathcal{J}_\epsilon\,:\, W^{1,p}_0(\Omega) \to \setR \cup
\set{\infty}$ explicitly as
\begin{align*}
  \Jeps(v) = \int\limits_\Omega \kappaeps (\vert\nabla v\vert)\d x -
  \langle f,v\rangle
\end{align*}
with $\kappaeps\,:\,\setR_{\ge 0}\to\setR$ given by
\begin{align*}
  \kappaeps(t) &:=
  \begin{cases}
    \frac{1}{2}\eps_-^{p-2}t^2 + (\frac{1}{p}-\frac{1}{2})\eps_-^p & \text{for }t\leq
    \epslower 
    \\
    \frac{1}{p}t^p & \text{for } \epslower\leq t\leq \epsupper 
    \\ 
    \frac{1}{2}\eps_+^{p-2}t^2 + (\frac{1}{p}-\frac{1}{2})\eps_+^{p} & \text{for }
    t\geq \epsupper.
  \end{cases}
\end{align*}
Note that $\frac 1p t^p \leq \kappa_\epsilon(t)$ for all~$t \geq 0$
and $\frac 1p t^p = \lim_{\epsilon \to [0,\infty]} \kappa_\epsilon(t)$ for
all~$t \geq 0$.  Since
$\kappa_\epsilon(t) \eqsim \epsilon_+^{p-2} t^2$ for large~$t$, we see
that $\mathcal{J}_\epsilon(v) < \infty$ if and only
if~$v \in W^{1,2}_0(\Omega)$.  Moreover,
$\lim_{\epsilon \to [0,\infty]} \mathcal{J}_\epsilon(v) =
\mathcal{J}(v)$
for all~$v \in W^{1,2}_0(\Omega)$ and
$\mathcal{J}(v) \leq \liminf_{\epsilon \to [0,\infty]}
\mathcal{J}_\epsilon(v)$ for all $v \in W^{1,p}_0(\Omega)$.

Based on the above observations it is natural to iteratively
minimize~$\mathcal{J}(v,a)$ alternating between~$v$ and
$a$. Certainly, we have also to increase the relaxation
interval~$\epsilon$. Thus our algorithm reads as follows:

\begin{algorithm}[H]
\SetAlgoLined
\TitleOfAlgo{The relaxed $p$-Ka{\v c}anov{} algorithm (variational formulation)}
\KwData{Given~$f \in (W^{1,p}_0(\Omega))^*$, $v_0 \in
  W^{1,2}_0(\Omega)$;}
\KwResult{Approximate solution of the $p$-Poisson
  problem~\eqref{eq:weakppoisson}; }

Initialize: $\eps_0=[\eps_{0,-},\eps_{0,+}] \subset (0,\infty)$,
$n=0$\;

\While{desired accuracy is not achieved yet}{
  Calculate $a_n$ by means of%
  \begin{align*}
    a_n&:= \argmin\limits_{a\,:\,\epsilon_- \leq a \leq
      \epsilon_+} \mathcal{J}(v_n, a);\;
  \end{align*}

  Calculate $v_{n+1}$ by means of
  \begin{align*}
    v_{n+1}&:= \argmin\limits_{v\in W_0^{1,2}(\Omega)} \mathcal{J}(v,
    a_n);\;
  \end{align*}
  
  Choose a new relaxation interval $\eps_{n+1} \supset \epsilon_n$\;

  Increase~$n$ by~$1$\;
}{
}
\end{algorithm}
This is just the algorithm given in the introduction written in
different form. 

\section{Convergence in the Relaxation Parameter}
\label{sec:epslimit}

In this section we show that the minimizers~$u_\epsilon$ of the relaxed
energy~$\mathcal{J}_\epsilon$ converge to the minimizer~$u$
of~$\mathcal{J}$ for $\epsilon \to [0,\infty]$ and derive an upper bound for the relaxation error.

Since $\mathcal{J}_\epsilon(v) \geq \mathcal{J}(v)$ and $\mathcal{J}$
is $W^{1,p}_0(\Omega)$ coercive, it follows
that~$\mathcal{J}_\epsilon$ is also coercive in
$W^{1,p}_0(\Omega)$. However,
$\mathcal{J}_\epsilon(v)<\infty$ requires~$v \in W^{1,2}_0(\Omega)$,
as we have seen above. Certainly, there is a gap between the
space~$W^{1,p}_0(\Omega)$ and~$W^{1,2}_0(\Omega)$. To close this gap
we need a finer analysis of the energies, which requires the use of
Orlicz spaces. We state in the following some standard results for
these spaces, see for example~\cite{KR61}.


A function $\phi:\setR_{\ge 0}\to\setR$ is called an N-function if and
only if there is a right-continuous, positive on the positive real line, and
non-decreasing function $\phi':\setR_{\ge 0}\to\setR$ with
$\phi'(0)=0$ and $\lim_{t \to \infty} \phi'(t)=\infty$ such that
$\phi(t) = \int_0^{t}\phi'(\tau)\d \tau$. An N-function is said to
satisfy the $\Delta_2$-condition if and only if there is a constant
$c>1$ such that $\phi(2t)\le c\,\phi(t)$. For an N-function satisfying
the $\Delta_2$-condition we define the Orlicz space to consist of
those functions~$v \in L_{\loc}^1(\Omega)$ with
$\int_\Omega \phi(|v|)\d x<\infty$. It becomes a Banach space with the
norm $\norm{f}_\phi := \inf \set{\gamma>0\,:\, \int_\Omega
  \phi(\abs{v}/\gamma)\,dx\leq 1}$.  The Orlicz-Sobolev space
$W^{1,\phi}(\Omega)$ then consists of those $v \in L^\phi$ such that
the weak derivative~$\nabla v$ is also in~$L^\phi$, equipped with the
norm $\norm{v}_\phi + \norm{\nabla v}_\phi$. The
space~$W^{1,\phi}_0(\Omega)$ denotes the subspace of those functions
from~$W^{1,\phi}(\Omega)$ with zero boundary trace, which coincides
with the closure of~$C^\infty_0(\Omega)$ in $W^{1,\phi}(\Omega)$.  For
example choosing $\phi(t):=\tfrac 1p t^p$ we have $L^\phi(\Omega) =
L^p(\Omega)$ and $W^{1,\phi}_0(\Omega) = W^{1,p}_0(\Omega)$.

The function~$\kappa_\epsilon$ cannot be
an N-function, since $\kappaeps(0)\neq 0$, . However, if we define
\begin{align}
  \label{eq:defphieps}
  \phieps(t) := \kappaeps(t) - \kappaeps(0),
\end{align}
then $\phieps$ is actually an N-function.  It can be verified
that~$\phi_\epsilon$ satisfies the~$\Delta_2$-condition with a
constant independent of~$\epsilon$.

Since $\phi_\epsilon(t) \eqsim \epsilon_+^{p-2} t^2$ for large~$t$
and~$\Omega$ is bounded, we have $L^{\phi_\epsilon}(\Omega) \eqsim
L^2(\Omega)$. However, the constant of the embedding $L^{\phi_\epsilon}(\Omega) \embedding L^2(\Omega)$
depends on~$\epsilon$, so this equivalence is not of much use.
Instead we use the chain of embeddings
\begin{align}
  \label{eq:L2toLphitoLp}
  L^2(\Omega) \embedding L^{\phi_\epsilon}(\Omega) \embedding L^p(\Omega),
\end{align}
with constants independent of~$\epsilon$. This follows from the fact
that the Simonenko indices of~$\phi_\epsilon$ are within~$[p,2]$. We refer the reader to, e.g., \cite[Chapter 2]{maxi16} for the details.

Since $\phieps$ is strictly convex and $\kappaeps(t) = \phieps(t) +
\kappaeps(0)$, the energy~$\Jeps$ admits a unique minimizer
$u_\epsilon \in W_0^{1,\phieps}(\Omega)$ whose Euler-Lagrange equation
is
\begin{align}
  \label{eq:equationOfueps}
  \int_\Omega (\eps_- \vee |\nabla u_\epsilon| \wedge
  \eps_+)^{p-2}\nabla u_\epsilon\cdot \nabla \xi \d x = \langle
  f,\xi\rangle \qquad \forall \xi \in W_0^{1,\phi_\epsilon}(\Omega).
\end{align}
At this we used that
\begin{align}
  \label{eq:phiepsprime}
  \frac{\phi_\epsilon'(t)}{t} &=  (\eps_- \vee t \wedge \epsilon_+)^{p-2}.
\end{align}
\begin{remark}
  \label{rem:no-upper-bound}
  Let us consider the special case~$\epsupper=\infty$. Then, the
  derivative of the truncated function reads
  $\phi_\epsilon'(t) = (\epslower \vee t)^{p-2} t$. A different
  modification, namely $(\epslower + t)^{p-2} t$, would lead to the
  so-called shifted N-function of $\frac{1}{p}t^p$, as introduced in
  more generality in~\cite{DE08}, which has similar properties as our
  truncation functions.  However, the version from this paper is more
  suitable for our energy relaxation, since it is closer to the
  original function~$\frac 1p t^p$ on the truncation
  interval~$\epsilon$ (the derivatives agree there). See the Appendix
  for more information on uniformly convex Orlicz functions and their
  shifted verions.
\end{remark}
\begin{lemma}
  \label{lem:phi-eps-uniformly-convex}
  The functions~$\phi$ and~$\phi_\epsilon$ are uniformly convex Orlicz
  functions in the sense of Section~\ref{sec:orlicz-functions} from
  the Appendix. The convexity constants are independent of~$\epsilon$.
\end{lemma}
\begin{proof}
  The uniform convexity of~$\phi$ follows
  from~$\frac{\phi''(t)\,t}{\phi'(t)}= (p-1)$ and
  Lemma~\eqref{eq:old-mon}. Now, for $t \in (\epsilon_-,\epsilon_+)$ we
  have $\frac{\phi_\epsilon''(t)\,t}{\phi\epsilon'(t)}= (p-1)$, while
  for $t \in (0,\infty) \setminus [\epsilon_-,\epsilon_+]$ we have
  $\frac{\phi_\epsilon''(t)\,t}{\phi_\epsilon'(t)} = 1$. Hence, the
  claim for~$\phi_\epsilon$ follows again by
  Lemma~\eqref{eq:old-mon}. \qed
\end{proof}

Since $W^{1,p}_0(\Omega)$ is the largest space,
see~\eqref{eq:L2toLphitoLp}, which contains both~$u_\epsilon$ and~$u$,
it is natural to consider all energies~$\mathcal{J}$ and
$\mathcal{J}_\epsilon$ as functionals on~$W^{1,p}_0(\Omega)$ with
possible value~$\infty$.

Let us recall that the goal of this section is to show
that~$u_\epsilon$ converges to~$u$ in $W^{1,p}_0(\Omega)$.  Since
$W^{1,p}_0(\Omega)$ is uniformly convex, strong convergence is a
consequence of weak convergence and norm convergence, or equivalently,
in this case, energy convergence
$\mathcal{J}(u_\epsilon) \to \mathcal{J}(u)$. It is possible to show
the weak convergence as well as that of the energy by means
of~$\Gamma$-convergence. Indeed, we will see in
Remark~\ref{rem:gammlimit} that $\mathcal{J}_\epsilon \to \mathcal{J}$
in the sense of~$\Gamma$-convergence. However, we will derive in the
following much stronger results that provide us with a precise rate of
convergence for the energies. This energy convergence implies strong
convergence of the sequence, see the proof of
Corollary~\ref{cor:convconstrained2}.

Let us turn to the convergence of the energies
$\mathcal{J}(u_\epsilon) \to \mathcal{J}(u)$ for $\epsilon \to
[0,\infty]$.  Since $\mathcal{J}_\epsilon$ is monotonically decreasing
with respect to~$\epsilon$, it follows from the minimizing properties
of~$u$ and~$u_\epsilon$ that
\begin{align}
  \label{eq:2}
  0 &\leq \mathcal{J}(u_\epsilon) - \mathcal{J}(u) \leq
  \mathcal{J}_\epsilon(u_\epsilon) - \mathcal{J}(u).
\end{align}
Therefore, it suffices to prove the stronger claim
\begin{align}
  \label{eq:1}
  \mathcal{J}_\epsilon(u_\epsilon) - \mathcal{J}(u) \to 0 \qquad
  \text{as~$\epsilon \to [0,\infty]$}.
\end{align}
In fact, we will later need this stronger estimate in the other
sections.

It follows from the minimizing property of~$u_\epsilon$ that
\begin{align*}
  \mathcal{J}_\epsilon(u_\epsilon) - \mathcal{J}(u) \leq
  \mathcal{J}_\epsilon(u) - \mathcal{J}(u).
\end{align*}
So it would be natural to estimate
$\mathcal{J}_\epsilon(u) - \mathcal{J}(u)$ in terms of~$\epsilon$
and~$u$. However, the solution~$u$ is unfortunately a~priori only
a~$W^{1,p}_0$-function, so~$\mathcal{J}_\epsilon(u)$ might be
infinity. Hence, we cannot assure that this difference is small. This
is only possible if we assume higher regularity of~$u$. In order to
treat arbitrary right-hand sides~$f \in (W^{1,p}_0(\Omega))^*$ at this
point, we have to use a much more subtle argument.  For this we need a
result from~\cite[Subsection~3.5]{DKS13} and
\cite[Theorem~2.7]{BDS16}, which allows to change~$u$ on a small set
such that it becomes a Lipschitz function. This technique is known as
the \emph{Lipschitz truncation technique}. Its origin goes back
to~\cite{AceF88}. As a tool we need the Hardy-Littlewood operator,
e.g. \cite{SM93},
 \begin{align*}
   (Mg)(x) &:= \sup_{r>0} \dashint_{B_r(x)} \abs{g}\,dx := \sup_{r>0} \frac{1}{|B_r(x)|}\int_{B_r(x)}|g|\d x 
 \end{align*}
where $\abs{B_r(x)}$ denote the Lebesgue measure of $B_r(x)$.
\begin{theorem}[Lipschitz trunction \cite{DKS13,BDS16}]\label{thm:liptrunc}
  Let $v\in W_0^{1,p}(\Omega)$ and for all $\lambda>0$ define
 \begin{align*}
   \mathcal{O}_\lambda := \mathcal{O}_\lambda(v) := \{x\in\Omega \,:\,
   M(\nabla v)(x)>\lambda\},
 \end{align*}
 Then, there
 exists an approximation $T_\lambda v\in W_0^{1,\infty}(\Omega)$ of
 $v$ with the following properties:
  \begin{enumerate}
  \item $\{v\neq T_\lambda v\}\subset \calO_\lambda$.
   \item $\Vert T_\lambda v\Vert_{L^p(\Omega)}\lesssim\Vert v\Vert_{L^p(\Omega)}$.
   \item $\Vert \nabla T_\lambda v\Vert_{L^p(\Omega)}\lesssim 
    \Vert \nabla v\Vert_{L^p(\Omega)}$.
  \item \label{itm:liptrunc-d} $\vert \nabla T_\lambda v\vert \lesssim \lambda
    \chi_{\calO_\lambda} + \vert \nabla v\vert\chi_{\Omega \setminus
      \calO_\lambda} \leq \lambda$ almost everywhere.
   \item \label{itm:liptrunc-e} $\norm{\nabla (v-T_\lambda v)}_{L^p(\Omega)} \lesssim
     \norm{\nabla v}_{L^p(\mathcal{O}_\lambda)}$.
   \item \label{itm:liptrunc-f} $v_\lambda \to v$ in
     $W^{1,p}_0(\Omega)$ as~$\lambda \to \infty$.
  \end{enumerate}
\end{theorem}
All our convergence results concerning the relaxation parameter $\eps$
are based on the following result, which shows how the energy
relaxation depends on the truncation interval~$\eps$.
\begin{theorem}
  \label{thm:keyepsconvest}
  The estimate
  \begin{align}
    \label{eq:keyepsconvest}
    \Jeps(u_\eps) - \calJ(u) \lesssim \epslower^p +
    \int\limits_{\calO_\lambda(u)}|\nabla u|^p \d x
 \end{align}
 holds for all $\lambda\le \epsupper/c_1$, where $c_1$ is the
 (hidden) constant from Theorem~\ref{thm:liptrunc}~\ref{itm:liptrunc-d}.
\end{theorem}
\begin{proof}
  Let $\lambda \leq \epsupper/c_1$ and let~$T_\lambda u$ be the
  Lipschitz truncation of~$u$. Then 
  \begin{align*}
    \abs{\nabla T_\lambda u} \leq c_1\,\lambda \leq \epsupper.
  \end{align*}
  Using the minimizing property of $u_\eps$ and the equation for~$u$
  we get
  \begin{align*}
    \Jeps(u_\eps) - \calJ(u) &\le \Jeps(T_\lambda u) - \calJ(u) =
    \int\limits_\Omega\left ( \kappaeps(|\nabla T_\lambda u|) - \tfrac 1p
    |\nabla u|^p \right )\d x - \langle f,T_\lambda u - u\rangle
    \\
    &= \int\limits_\Omega \left ( \kappaeps(|\nabla T_\lambda u|) - \tfrac 1p
    |\nabla u|^p \right ) \d x - \int_\Omega |\nabla
    u|^{p-2}\nabla u\cdot\nabla(T_\lambda u - u) \d x.
  \end{align*}
  Using~$\abs{\nabla T_\lambda u} \leq \epsupper$, $\kappaeps(t)=\tfrac
  1p t^p$ for $t \in [\epslower,\epsupper]$, $\kappaeps(t) \leq \tfrac
  1p \epslower^p$ for $t \in [0, \epslower]$, and ~$T_\lambda u=u$
  outside of~$\mathcal{O}_\lambda$ we get
  \begin{align*}
    \kappaeps(|\nabla T_\lambda u|) - \tfrac 1p |\nabla u|^p &\leq
    \begin{cases}
      \frac 1{p} \epslower^p &\qquad \text{on } \set{\abs{\nabla
          T_\lambda u} \leq \epslower},
      \\
      0 &\qquad \text{on } \big(\Omega \setminus
      \mathcal{O}_\lambda\big) \cap \set{\abs{\nabla T_\lambda u} >
        \epslower},
      \\
      \frac 1{p} \abs{\nabla T_\lambda u}^p &\qquad \text{on }
      \mathcal{O}_\lambda  \cap \set{\abs{\nabla T_\lambda u} >
        \epslower}.
    \end{cases}
  \end{align*}
  This, the previous estimate and
  Theorem~\ref{thm:liptrunc}~\ref{itm:liptrunc-e} imply
  \begin{align*}
    \Jeps(u_\eps) - \calJ(u) 
    &\leq \abs{\Omega}\,\tfrac 1p \epslower^p + \int\limits_{\calO_\lambda(u)}
    \tfrac 1p |\nabla T_\lambda u|^p \d x + \int\limits_{\calO_\lambda(u)} |\nabla
    u|^{p-1} \abs{\nabla(T_\lambda u - u)}\d x
    \\
    &\lesssim \epslower^p + \int\limits_{\calO_\lambda(u)} |\nabla
    T_\lambda u|^p \d x + \int\limits_{\calO_\lambda(u)} |\nabla u|^p \d
    x + \int\limits_{\calO_\lambda(u)} |\nabla (T_\lambda u-u)|^p \d x
    \\
    &\lesssim \epslower^p + \int\limits_{\calO_\lambda(u)} |\nabla
    u|^p \d x. 
  \end{align*}
  This proves the claim. \qed
\end{proof}
\begin{corollary}
  \label{cor:convconstrained}
  $\Jeps(u_\eps) \to \calJ(u)$ and $\mathcal{J}(u_\eps) \to \calJ(u)$
  as $\eps\to [0,\infty]$.
\end{corollary}
\begin{proof}
  Due to~\eqref{eq:2} it suffices to prove~$\Jeps(u_\eps) \to
  \calJ(u)$.  Consider the right-hand side of~\eqref{eq:keyepsconvest}
  with $\lambda := \epsupper/c_1$. The first term goes to zero
  as~$\epslower \to 0$. Now consider the second
  term. Since~$\calO_\lambda(u) \subset \set{M(\nabla u) >
    \lambda}$ and~$\nabla u\in L^p(\Omega)$ we get by the
  weak~$L^p$-estimate of the maximal operator $|\calO_\lambda(u)|
  \lesssim \lambda^{-p} \norm{\nabla u}_p^p$. Therefore
  $|\calO_\lambda(u)|\to 0$ as~$\epsupper \to \infty$, which implies
  $\int_{\calO_\lambda(u)} |\nabla u|^p \d x \to 0$ as $\epsupper \to
  \infty$.  \qed
\end{proof}
Before we continue we need the following natural quantities,
see~\cite{DE08}. \begin{definition}\label{def:AandV} For $P\in\setR^d$
  we define
  \begin{align*}
    \Aeps(P):= \begin{cases}\frac{\phieps'(\vert P\vert)}{\vert
        P\vert}P &\text{if } P\neq 0\\ 0&\text{if
      }P=0\end{cases}\quad\text{and}\quad
    \Veps(P):=\begin{cases}\sqrt{\frac{\phieps'(\vert P\vert)}{\vert
          P\vert}}P&\text{if }P\ne 0\\0&\text{if }P=0.\end{cases}
  \end{align*}
  Moreover, by $A := A_{[0,\infty]}$ and $V := V_{[0,\infty]}$ we
  denote the unrelaxed versions.
\end{definition}
The following two lemmas are modifications of similar results of
\cite[Lemma~3]{DE08} and~\cite[Lemma~16]{DieK08}. In fact, they follow
from the properties of uniformly convex Orlicz functions; see
Section~\ref{sec:orlicz-functions} from the Appendix for more details.
\begin{lemma}\label{lem:phistructure}
  For $P,Q\in\setR^d$
  \begin{align*}
    (\Aeps(P)-\Aeps(Q))\cdot (P-Q)\eqsim \frac{\phieps'(|P|\vee
    |Q|)}{|P|\vee |Q|}\vert P-Q\vert^2 \eqsim \abs{V_\epsilon(P) -
    V_\epsilon(Q)}^2. 
  \end{align*}
  where the constants can be chosen independently of $\eps$.
\end{lemma}
\begin{proof}
  This follows directly from the uniform convexity of~$\phi_\epsilon$,
  see Lemma~\ref{lem:phi-eps-uniformly-convex}, Lemma~\ref{lem:hammer}
  and Lemma~\ref{lem:phi-shift-equiv}. \qed
\end{proof}
\begin{lemma}
  \label{lem:JAV}
  The following estimates hold for arbitrary $v\in
  W_0^{1,\phieps}(\Omega)$ and $u_\eps$ being the minimizer of
  $\Jeps$:
  \begin{align*}
    \Jeps(v)-\Jeps(u_\eps)&\le \int\limits_\Omega (\Aeps(\nabla
    v)-\Aeps(\nabla u_\eps))\cdot\nabla (v-u_\eps)\d x
    \\
    &\lesssim \int\limits_\Omega |\Veps(\nabla v)-\Veps(\nabla
    u_\eps)|^2\d x
    \\
    &\lesssim \Jeps(v)-\Jeps(u_\eps) .
  \end{align*}
  In particular, for the case where $\varepsilon=[0,\infty]$ the
  statement actually implies also
  \begin{align*}
    \mathcal J(v)-\mathcal J(u)& \eqsim \int\limits_\Omega |V(\nabla v)-V(\nabla
                                 u)|^2\d x.
  \end{align*}
\end{lemma}
\begin{proof}
  This is just Lemma~\ref{lem:JAVpre} applied to~$\phi_\epsilon$. \qed
\end{proof}

We are now prepared to show the convergence of minimizers~$u_\epsilon$
of~$\mathcal{J}_\epsilon$ to~$u$.
\begin{corollary}
  \label{cor:convconstrained2}
  $u_\epsilon \to u$ in $W^{1,p}_0(\Omega)$ as $\eps\to[0,\infty]$.
\end{corollary}
\begin{proof}
  Due to Corollary~\ref{cor:convconstrained} we have
  $\mathcal{J}(u_\epsilon) - \mathcal{J}(u) \to 0$ as~$\epsilon \to
  [0,\infty]$. Now Lemma~\ref{lem:JAV} for the case where~$\epsilon =[0,\infty]$ 
  and $v =u_\varepsilon$  implies $V(\nabla u_\epsilon) \to V(\nabla u)$
  in~$L^2(\Omega)$. It follows from the shift-change-lemma, see
  Corollary~\ref{cor:shift-change-new} or 
  \cite[Corollary~26]{DieK08},  that for all~$\delta>0$ there
  exists~$c_\delta>0$ such that
  \begin{align*}
    \abs{\nabla u - \nabla u_\epsilon}^p &\leq c_\delta \abs{V(\nabla
      u_\epsilon) - V(\nabla u)}^2 + \delta \abs{\nabla u}^p.
  \end{align*}
  This and $V(\nabla u_\epsilon) \to V(\nabla u)$
  in~$L^2(\Omega)$ implies $\nabla u_\epsilon \to \nabla u$
  in~$L^p(\Omega)$. \qed
\end{proof}

\begin{remark}[$\Gamma$-convergence]
  \label{rem:gammlimit}
  It is also possible to deduce~$\mathcal{J}_\epsilon(u_\epsilon) \to
  \mathcal{J}(u)$ and $u_\epsilon \to u$ in~$W^{1,p}_0(\Omega)$ by
  means of~$\Gamma$-convergence: As the underlying topological space
  we choose $W^{1,p}_0(\Omega)$ equipped with the weak topology. Then
  the Lipschitz truncation provides a recovery sequence for~$v \in
  W^{1,p}_0(\Omega)$ implying $\Gammalim_{\epsilon \to
    [0,\infty]} \mathcal{J}_\epsilon = \mathcal{J}$.  Indeed, it
  follows as in the  proof of Theorem~\ref{thm:keyepsconvest} that
  for all~$v \in W^{1,p}_0(\Omega)$
  \begin{align*}
    \bigabs{\mathcal{J}_\epsilon(T_{\epsupper/c_1} v) -
      \mathcal{J}(v)} &\lesssim \epslower^p +
    \int_{\mathcal{O}_{\epsupper/c_1}(v)} \abs{\nabla v}^p\,dx +
    \bigabs{ \skp{f}{T_{\epsupper/c_1} v -v}}.
  \end{align*}
  So the properties of the Lipschitz truncation, see
  Theorem~\ref{thm:liptrunc}~\ref{itm:liptrunc-f}, imply that the
  right-hand side goes to zero as~$\epsilon \to [0,\infty]$. Hence,
  $T_{\epsupper/c_1} v$ is a recovery sequence of~$v$. 

  Moreover, $\mathcal{J}_\epsilon \geq \mathcal{J}$, so the standard
  theory of $\Gamma$-convergence  \cite{MR1201152,MR1968440} proves~$u_\epsilon \weakto u$
  in~$W^{1,p}_0(\Omega)$ and $\mathcal{J}_\epsilon(u_\epsilon) \to
  \mathcal{J}(u)$ for~$\epsilon \to [0,\infty]$. The uniform convexity
  of~$W^{1,p}_0(\Omega)$ implies $u_\epsilon \to u$
  in~$W^{1,p}_0(\Omega)$.

  To our knowledge this is the first time that the Lipschitz
  truncation is used to construct a recovery sequence for the $\Gamma-\lim\sup$ in a $\Gamma$-convergence argument related to energies on $W^{1,p}_0(\Omega)$.
\end{remark}

Up to now, we discussed the convergence of~$u_\epsilon \to u$ without
any additional assumptions on the data~$f \in
(W_0^{1,p}(\Omega))^\ast$ and the domain~$\Omega$. If~$f$ is more
regular and~$\partial \Omega$ is suitably smooth, then we obtain
specific rates for the convergence. 
The rates of convergence will follow from the regularity of~$\nabla u$
in terms of the weak-$L^q$ spaces $L^{q,\infty}(\Omega)$, which
consists of all functions~$v$ such that
\begin{align*}
  \|v\|_{L^{q,\infty}(\Omega)} := \sup\limits_{t>0} \norm{t\,
    \chi_{\set{\abs{v}>t}}}_{L^q(\Omega)} < \infty.
\end{align*}

\begin{lemma}\label{lem:convLorentz}
  Let $\nabla u\in L^{q,\infty}(\Omega)$ for some $q>p$. Then,
  \begin{align*}
    \Jeps(u_\eps)- \calJ(u) \lesssim \epslower^p +
    \epsupper^{-(q-p)}\,\|\nabla u\|_{L^{q,\infty}(\Omega)}^q.
  \end{align*}
\end{lemma}
\begin{proof}
  First note that $M:L^{q,\infty}(\Omega)\to L^{q,\infty}(\Omega)$ is
  bounded. This follows for example by extrapolation theory,
  see~\cite[Theorem~1.1]{CruMarPer04}. In particular,
  \begin{align*}
    \lambda \abs{\calO_\lambda(u)}^{\frac 1q} \eqsim
    \norm{\lambda\,\chi_{\set{M(\nabla u)> \lambda}}}_{L^q(\Omega)}
    \leq \norm{M(\nabla u)}_{L^{q,\infty}(\Omega)} \lesssim
    \norm{\nabla u}_{L^{q,\infty}(\Omega)}.
  \end{align*}
  Moreover, let $L^{q,1}(\Omega)$ denote the usual Lorentz space,
  which consists of functions~$v$ such  that
  \begin{align*}
    \|v\|_{L^{q,1}(\Omega)} := q\int\limits_0^\infty |\{|v|>
    t\}|^\frac{1}{q}\d t = q \int\limits_0^\infty \norm{t\,
      \chi_{\set{\abs{v}>t}}}_{L^q(\Omega)} \frac{dt}{t} < \infty.
  \end{align*}
  Since $(L^{s',1})^* =L^{s,\infty}$ for $1<s<\infty$ and
  $\frac{1}{s}+\frac{1}{s'}=1$ we obtain
  \begin{align*}
    \int\limits_{\calO_\lambda(u)} |\nabla u|^p \d x &\lesssim
    \bignorm{|\nabla u|^p}_{L^{\frac{q}{p},\infty}(\Omega)}\|
    \chi_{\calO_\lambda(u)}\|_{L^{\frac{q}{q-p},1}(\Omega)}
    \\
    &\eqsim \|\nabla u\|_{L^{q,\infty}(\Omega)}^p
    \abs{\mathcal{O}_\lambda(u)}^{\frac{q-p}{q}}
    \\
    &\lesssim \|\nabla u\|_{L^{q,\infty}(\Omega)}^q \,\lambda^{-(q-p)},
  \end{align*}
  where~$\abs{\mathcal{O}_\lambda}$ denotes the Lebesgue measure
  of~$\mathcal{O}_\lambda$.  Applying Theorem~\ref{thm:keyepsconvest}
  with $\lambda := \epsupper/c_1$ yields the statement. \qed
\end{proof}
To exemplify the consequences of
Lemma~\ref{lem:convLorentz} we combine it with the regularity results
of~\cite{CM10} and~\cite{E02}:
\begin{theorem}[\cite{CM10}, Theorems~1.3 and 1.4]\label{thm:regCM}
  Let $\Omega\subset\setR^d$ be convex or let its boundary $\partial\Omega\in W^2
  L^{d-1,1}$ (for example $\partial \Omega \in C^2$ suffices\footnote{The condition~$\partial \Omega \in W^2
    L^{d-1,1}$ means that the weak Hessians of the local maps
    characterizing the boundary are in the
    Lorentz space~$L^{d-1,1}$}) and additionally $f\in L^{d,1}(\Omega)$. Then
  $\nabla u \in L^\infty(\Omega)$.
\end{theorem}
\begin{theorem}[\cite{E02}, (4.3)]\label{thm:regE}
  Let~$\Omega$ be a polyhedral domain where the inner angle is
  strictly less than $2\pi$ and $f\in L^{p'}(\Omega)$ and $\frac{1}{p}
  + \frac{1}{p'}=1$. Then $\nabla u \in L^{\frac{pd}{d-1},\infty}(\Omega)$.
\end{theorem}
\begin{proof}
  Actually, it is proven in~\cite{E02}~(4.3) that $\abs{\nabla
    u}^{\frac p2} \in
  \calN^{\frac 12,2}(\Omega)$ (\Nikolskii{} space). Now, one can use
  the embedding~$\calN^{\frac 12,2}(\Omega)\hookrightarrow
  L^{\frac{2d}{d-1},\infty}(\Omega)$ of Lemma~\ref{lem:embedding} in
  the Appendix. \qed
\end{proof}
\begin{corollary}
  \label{cor:regCM}
  Under the assumptions of Theorem~\ref{thm:regCM} we have
  \begin{align*}
    \Jeps(u_\eps)- \calJ(u) \lesssim \epslower^p.
  \end{align*}
\end{corollary}
\begin{proof}
  Since $\nabla u \in L^\infty(\Omega)$, we have $M(\nabla u)\in
  L^\infty(\Omega)$ and so for $\lambda:=\epsupper/c_1$ and
  $\epsupper$ large enough, $\calO_\lambda(u)=\emptyset$. Hence, 
  Theorem~\ref{thm:keyepsconvest} implies the estimate. \qed
\end{proof}
\begin{corollary}
  \label{cor:regE}
  Under the assumptions of Theorem~\ref{thm:regE} we have
  \begin{align*}
    \Jeps(u_\eps)- \calJ(u) \lesssim \epslower^p + \epsupper^{-\frac{p}{d-1}} 
  \end{align*}
\end{corollary}
\begin{proof}
  Since $\nabla u\in L^{\frac{pd}{d-1},\infty}(\Omega)$, an
  application of Lemma~\ref{lem:convLorentz} finishes the proof. \qed
\end{proof}
\begin{remark}
  The choice~$f=0$ and hence~$u=0$ gives $\mathcal{J}_\epsilon(u) =
  \kappa_\epsilon(0) \abs{\Omega} \eqsim \epslower^p$. This shows that
  the estimate in Corollary~\ref{cor:regCM} is sharp.
\end{remark}

\section{\texorpdfstring{Convergence of the Ka{\v c}anov-Iteration}{Convergence of the Kacanov-Iteration}}
\label{sec:nlimit}

In this section we study the convergence of the Ka{\v c}anov-iteration
for fixed relaxation parameter~$\epsilon =[\epslower,\epsupper]$.  In
particular, for $v_0\in W_0^{1,2}(\Omega)$ arbitrary we calculate
recursively~$v_{n+1}$ by
\begin{align}
  \label{eq:3}
  \int\limits_\Omega (\constrained{|\nabla v_n|})^{p-2}\nabla
  v_{n+1}\cdot \nabla \xi \d x = \langle f,\xi\rangle\quad\forall
  \xi\in W_0^{1,2}(\Omega).
\end{align}
We will show that $v_n$ converges to the minimizer~$u_\epsilon$ of the
relaxed energy~$\Jeps$. In particular, we show exponential decay of
the energy error~$\Jeps(v_n)-\Jeps(u_\epsilon)$.
The proof is based on the following estimate, proved below.
\begin{theorem}\label{thm:convalg}
  There is a constant $c_K>1$ such that
  \begin{align*}
    \Jeps(v_n)-\Jeps(v_{n+1}) \geq \delta\,(\Jeps(v_n)-\Jeps(u_\eps))
  \end{align*}
  holds for $\delta := \tfrac 1{c_K} (\tfrac \epslower \epsupper)^{2-p}$.
\end{theorem}
This theorem says that in each iteration we reduce the energy by a
certain part of the remaining energy error. This implies
\begin{align}
  \label{eq:Jdiffreduction}
  \begin{aligned}
    \Jeps(v_{n+1})-\Jeps(u_\eps) &= \big( \Jeps(v_n)-\Jeps(u_\eps)
    \big) - \big( \Jeps(v_n)-\Jeps(v_{n+1}) \big)
    \\
    & \leq (1-\delta)\, \big( \Jeps(v_n)-\Jeps(u_\eps) \big).
  \end{aligned}
\end{align}
As a direct consequence we will obtain the following exponential
convergence result.
\begin{corollary}\label{cor:convalg}
  There is a constant $c_K>1$ such that
  \begin{align*}
    \Jeps(v_n)-\Jeps(u_\eps)\leq (1-\delta)^n(\Jeps(v_0)-\Jeps(u_\eps)) .
  \end{align*}
  holds for $\delta := \tfrac 1{c_K} (\tfrac \epslower \epsupper)^{2-p}$.
\end{corollary}
Let us get to the proof of Theorem~\ref{thm:convalg}.
\begin{proof}[Proof of Theorem~\ref{thm:convalg}]
  Using Lemma~\ref{lem:JAV}, the equation \eqref{eq:equationOfueps}
  for $u_\eps$, the equation~\eqref{eq:3} for $v_{n+1}$, and Young's
  inequality (see Remark~\ref{rem:young}) we get, for arbitrary $\gamma > 0$,
  \begin{align*}
    \Jeps(v_n)-\Jeps(u_\eps) &\leq \int\limits_\Omega
    (\Aeps(\nabla v_n)-\Aeps(\nabla u_\eps))\cdot\nabla (v_n-u_\eps)\d x
    \\
    &= \int\limits_\Omega \tfrac{\phieps'(\vert \nabla
      v_n\vert)}{\vert\nabla v_n\vert}\nabla( v_n - v_{n+1})\cdot
    \nabla (v_n-u_\eps)\d x
    \\
    &\leq \tfrac{1}{\gamma}\underbrace{ \tfrac 12 \int\limits_\Omega
      \tfrac{\phieps'(\vert \nabla v_n\vert)}{\vert\nabla
        v_n\vert}\vert\nabla( v_n - v_{n+1})\vert^2\d x}_{=:I}
    \\
    &\hspace*{1cm}+ \gamma\,\underbrace{\tfrac 12 \int\limits_\Omega
      \tfrac{\phieps'(\vert \nabla v_n\vert)}{\vert\nabla
        v_n\vert}\vert \nabla (v_n-u_\eps)\vert^2\d x}_{=:II} .
  \end{align*}
  Let us define
  \begin{align*}
    \mathcal{J}_\epsilon(v, a) := \calJ\left(v,\constrained{a}\right).
  \end{align*}
  For the first term~$I$ we calculate with the equation~\eqref{eq:3} for $v_{n+1}$
  \begin{align*}
    I &= \tfrac 12 \int\limits_\Omega \tfrac{\phieps'(\vert \nabla
      v_n\vert)}{\vert\nabla v_n\vert}\vert\nabla( v_n -
    v_{n+1})\vert^2\d x
    \\
    &=\Jeps(v_n,\abs{\nabla v_n})-\Jeps(v_{n+1},\abs{\nabla v_n})
    \\
    &\leq \Jeps(v_n,\abs{\nabla v_n})-\Jeps(v_{n+1},\abs{\nabla
      v_{n+1}})
    \\
    &= \Jeps(v_n)-\Jeps(v_{n+1}).
  \end{align*}
  To establish the inequality above, we used the fact that 
  $$
  a_{n+1} = \constrained{\abs{\nabla
      v_{n+1}}} = \argmin\limits_{a\,:\,\epsilon_- \leq a \leq
    \epsilon_+} \mathcal{J}(v_{n+1}, a),
  $$
  and
  \begin{align*}
    \Jeps(v_{n+1},\abs{\nabla v_{n+1}}) &= \mathcal
                                          J(v_{n+1},\constrained{\abs{\nabla
                                          v_{n+1}}}) 
    \\
                                        &\leq 
                                          \mathcal J(v_{n+1},\constrained{\abs{\nabla v_{n}}})
                                          =\Jeps(v_{n+1},\abs{\nabla v_n}).
  \end{align*}
  For the second term~$II$ we use
  $\epsupper^{p-2}\le \tfrac{\phieps'(t)}t \le \epslower^{p-2}$,
  implying
  $\tfrac{\phieps'(t)}t \leq
  (\tfrac{\epsupper}{\epslower})^{2-p}\tfrac{\phieps'(s)}s$, for any
  $s,t \geq 0$, Lemma~\ref{lem:phistructure} and Lemma~\ref{lem:JAV}
  to get
  \begin{align*}
    II &\leq \tfrac 12 (\tfrac{\epsupper}{\epslower})^{2-p}\int\limits_\Omega
    \tfrac{\phieps'(\vert \nabla v_n \vert\vee \vert \nabla
      u_\eps\vert)}{\vert \nabla v_n \vert\vee \vert \nabla
      u_\eps\vert}\vert \nabla (v_n-u_\eps)\vert^2\d x
    \\
    &\leq c\, (\tfrac{\epsupper}{\epslower})^{2-p}\,(\Jeps(v_n)-\Jeps(u_\eps)) .
  \end{align*}
  Putting all estimates together we get
  \begin{align*}
    \gamma(1-c\gamma(\tfrac{\epsupper}{\epslower})^{2-p})\, (
    \Jeps(v_n)-\Jeps(u_\eps))\le \Jeps(v_n)-\Jeps(v_{n+1}).
  \end{align*}
  Now, $\max_{\gamma>0}
  \gamma(1-c\gamma(\tfrac{\epsupper}{\epslower})^{2-p}) =
  \tfrac{1}{4c} (\tfrac{\epslower}{\epsupper})^{2-p}$ yields the
  statement. \qed
\end{proof}
\begin{example}[Peak function]\label{ex:peak}
  Let $\Omega:= B_1(0)$ and $f(x)=-\div(\tfrac{x}{|x|})$. Then
  $f\notin L^1(\Omega)$ but still
  $f \in (W_0^{1,1}(\Omega))^\ast \subset (W_0^{1,p}(\Omega))^\ast$
  . Then the minimizer of~$\mathcal{J}$ is given by
  $u(x) = 1-\abs{x}$, which look like a peak. Since
  $|\nabla u|\equiv 1$, the factor $\abs{\nabla u}^{p-2}$ in the
  $p$-Laplace operator does not appear for the minimizer. So in this case~$u$
  also minimizes every~$\Jeps$ as long as $\epslower \leq 1$ and
  $\epsupper \geq 1$. This follows from
  \begin{align*}
    \Jeps(u) = \calJ(u)\le \calJ(v)\le \Jeps(v)
  \end{align*}
  for all $v\in W_0^{1,\phieps}(\Omega)$.

  Let us see how our algorithm performs with the starting value~$v_0:=0$.
  It is easy to see that $v_n = \alpha_n u$ with
  \begin{align}
    \label{eq:4}
    \alpha_0 := 0 \quad\text{and}\quad
    \alpha_{n+1}:=(\constrained{\alpha_n})^{2-p}.
  \end{align}
  Since $p\in(1,2)$ one can show $\alpha_ n = \epslower^{(2-p)^n}$ by
  induction and
  \begin{align*}
    \Jeps(v_n)-\calJ(u)
    &= \tfrac 1p |B_1(0)|(\epslower^{p(2-p)^n} -1 -
    p(\epslower^{(2-p)^n} -1)) .
  \end{align*}
  Note that
  \begin{align*}
    \tfrac 1p t^p - \tfrac 1p - (t-1) \leq \tfrac{p-1}{2} (\ln(t))^2
    \qquad \text{for all }t \in (0,1].
  \end{align*}
  Moreover,
  \begin{align} \label{asymp}
    \frac{\tfrac 1p t^p - \tfrac 1p - (t-1)}{\tfrac{p-1}{2}
      (\ln(t))^2} &\to 1 \qquad \text{as } t \nearrow 1.
  \end{align}
  This estimate with $s:=(2-p)^n \in (0,1]$ and $t := \epslower^s \in
  (0,1]$ gives

  \begin{align}\label{estim}
    \Jeps(v_n)- \calJ(u)\le \tfrac{1}2|B_1(0)|(p-1)\ln(\epslower)^2(2-p)^{2n}
  \end{align}
  is sharp. Indeed, the energy differences $\calJ(v_n)-\calJ(u) =
  \Jeps(v_n)-\calJ(u)$ asymptotically behave like
  $\tfrac{1}2|B_1(0)|(p-1)\ln(\epslower)^2(2-p)^{2n}$ for large
  $n$, in view of \eqref{asymp}.

  This shows that it is impossible to get an energy
  reduction as in~\eqref{eq:Jdiffreduction} with~$\delta$ independent
  of~$\epsilon$. Indeed, Corollary~\ref{cor:convalg} would imply
  \begin{align*}
    \mathcal{J}_\epsilon(v_n) - \mathcal{J}_\epsilon(u_\epsilon) \leq
    (1-\delta)^n \big(\mathcal{J}_\epsilon(v_0) -
    \mathcal{J}_\epsilon(u_\epsilon)\big) \leq (1-\delta)^n
    \mathcal{J}_1(v_0),
  \end{align*}
  which contradicts the above asymptotic estimate \eqref{estim}.

  Nevertheless, our asymptotic shows that in this particular case
  \begin{align*}
    \Jeps(v_n)-\Jeps(u_\eps)\leq
    c_\epsilon\,(1-\delta)^n
  \end{align*}
  with $1-\delta = (2-p)^2 <1$ independent of~$\epsilon$. Therefore, it
  remains open if such an estimate holds in the general case.
\end{example}
\begin{remark}
  \label{rem:peak3}
  Although our considerations are all under the assumption~$1<p\leq 2$
  it is interesting to check how the algorithm performs in the
  case~$p> 2$ for our Example~\ref{ex:peak}.

  If $p\ge 3$ and $\epsupper := \tfrac{1}{\epslower}$ for some
  $\epslower<1$, then it follows from~\eqref{eq:4} that
  \begin{align*}
    \alpha_0 = 0\quad \text{and}\quad \alpha_n =
    \epslower^{(-1)^n(p-2)} \quad\text{ for } n\ge 1 .
  \end{align*}
  Therefore, the Ka\v{c}anov iteration does not converge as $p\ge
  3$. 

  If $p \in (2,3)$ and $\epsupper := \tfrac{1}{\epslower}$, then it
  follows from~\eqref{eq:4} that
  \begin{align*}                %
    \alpha_0 = 0\quad \text{and}\quad \alpha_n = \epslower^{(2-p)^n}
    \quad\text{ for } n\ge 1
  \end{align*}
  and~$v_n$ still converges to~$u$.
\end{remark}

\section{Algebraic Rate}
\label{sec:algebraic}%

As we learned in the last section the Ka\v{c}anov iteration converges
for fixed~$\epsilon$, but the rate depends badly on the choice of the
relaxation interval $\eps=[\epslower,\epsupper]$. Furthermore, we have
algebraic convergence of the error $\Jeps(u_\eps)- \calJ(u)$ induced
by the relaxation. We will combine these results to deduce an
algebraic rate of the full error $\mathcal{J}_{\epsilon_n}(v_n)
-\mathcal{J}(u)$ in terms of~$n$ for a specific
predefined choice of~$\epsilon_n$. To achieve our goal
we will use that $|\nabla u|\in L^{q,\infty}(\Omega)$ for some
$q>p$, which is justified by Theorems~\ref{thm:regCM}
and~\ref{thm:regE}.

Let us consider a sequence of solutions created by our relaxed
$p$-Ka{\v c}anov{} algorithm. In particular, $\epsilon_n=[\epsilon_{n,-},\epsilon_{n,+}]$ is now an increasing
sequence of intervals. Then exactly as in Theorem~\ref{thm:convalg} we get the
following estimate.
\begin{theorem}
  \label{thm:convalg2}
  There is a constant $c_K>1$ such that
  \begin{align*}
    \mathcal{J}_{\epsilon_n}(v_n)-\mathcal{J}_{\epsilon_n}(v_{n+1})
    &\geq \delta_n 
      \,(\mathcal{J}_{\epsilon_n}(v_n)-\mathcal{J}_{\epsilon_n}(u_{\eps_n}))
  \end{align*}
  holds for $  \delta_n := \tfrac 1{c_K}
  (\tfrac{\epsilon_{n,-}}{\epsilon_{n,+}})^{2-p}$.
\end{theorem}
Since~$\epsilon_n \subset \epsilon_{n+1}$, we
have~$\mathcal{J}_{\epsilon_{n+1}} \leq
\mathcal{J}_{\epsilon_n}$. This and Theorem~\ref{thm:convalg2} imply
\begin{align*}
  \lefteqn{\mathcal{J}_{\epsilon_{n+1}}(v_{n+1}) - \mathcal{J}(u)
  }\quad 
  &
  \\
  &\leq \mathcal{J}_{\epsilon_n}(v_{n+1}) - \mathcal{J}(u)
  \\
  &= \big(\mathcal{J}_{\epsilon_n}(v_n) - \mathcal{J}(u)\big) -
    \big( \mathcal{J}_{\epsilon_n}(v_n)-
    \mathcal{J}_{\epsilon_n}(v_{n+1}) \big)
  \\
  &\leq \big(\mathcal{J}_{\epsilon_n}(v_n) - \mathcal{J}(u)\big) -
    \delta_n \big( \mathcal{J}_{\epsilon_n}(v_n)-
    \mathcal{J}_{\epsilon_n}(u_{\epsilon_n}) \big)
  \\
  &= (1- \delta_n) \big(\mathcal{J}_{\epsilon_n}(v_n) -
    \mathcal{J}(u)\big) + \delta_n \big(
    \mathcal{J}_{\epsilon_n}(u_{\epsilon_n}) - \mathcal{J}(u) \big).
\end{align*}
Now, Lemma~\ref{lem:convLorentz} and $|\nabla u|\in
L^{q,\infty}(\Omega)$ ensure the existence of $c_R>0$ such that
\begin{align}\label{eq:RelaxError}
  \Jeps(u_\eps)- \calJ(u) \le c_R(\epslower^p + \epsupper^{-(q-p)}) .
\end{align}
This and the previous estimate therefore imply
\begin{align}
  \label{eq:decay-J-plus-extra}
  \mathcal{J}_{\epsilon_{n+1}}(v_{n+1}) - \mathcal{J}(u) 
  &\leq (1- \delta_n) 
    \big(\mathcal{J}_{\epsilon_n}(v_n) - 
    \mathcal{J}(u)\big) + \delta_n c_R(\epsilon_{n,-}^p + \epsilon_{n,+}^{-(q-p)}).
\end{align}
Without the last term
$\delta_{n+1} c_R(\epsilon_{n,-}^p + \epsilon_{n,+}^{-(q-p)})$ we
would get a reduction of the
error~$\mathcal{J}_{\epsilon_n}(v_n) -\mathcal{J}(u)$ by the
factor~$(1-\delta_n)$. On the other hand this last term is small
if~$\epsilon_{n,-} \to 0$ and~$\epsilon_{n,+} \to \infty$, so it
should not bother too much.  Nevertheless, the reduction
factor~$(1-\delta_n)$ tends to~$1$ if~$\epsilon_{n,-} \to 0$ and
$\epsilon_{n,+} \to \infty$.  The idea however is the following:
if~$\delta_n$ goes to zero slowly, then the
product~$\prod_{i=1}^n (1-\delta_i)$ still tends to zero
algebraically.

Let us be more precise. We define another relaxed
energy~$\mathcal{G}_n$ by
\begin{align}\label{eq:calGdef}
  \calG_n := \Jepsn(v_n) + K_1\, (\eps_{-,n}^p + \eps_{+,n}^{-(q-p)})
  \quad \text{and}\quad \calG_\infty := \calJ(u) , 
\end{align}
where $K_1>0$ will be determined below.  Moreover, choose
\begin{align}
  \text{$\alpha,\beta>0$ with
  $\alpha+\beta \leq \tfrac{1}{2-p}$}
\end{align}
and define
\begin{align}
  \eps_n := [(n+1)^{-\alpha},(n+1)^\beta].  
\end{align}
 Then
\begin{align}
  \delta_n 
  &= 
    \tfrac 1{c_K}
    \big(\tfrac{\epsilon_{n,-}}{\epsilon_{n,+}}\big)^{2-p} 
    =
    \tfrac 1{c_K}
    ((n+1)^{-\alpha-\beta})^{2-p} \geq \tfrac 1{c_K} \tfrac 1{n+1}.
\end{align}
In particular, the algorithm with this choice of~$\epsilon_n$ reads as
follows:

\begin{algorithm}[H]
\SetAlgoLined
\TitleOfAlgo{The relaxed $p$-Ka{\v c}anov{} algorithm with algebraic rate}
\KwData{Given~$f \in (W^{1,p}_0(\Omega))^*$, $v_0 \in
  W^{1,2}_0(\Omega)$;}
\KwResult{Approximate solution of the $p$-Poisson
  problem~\eqref{eq:weakppoisson}; }
Initialize: $n=0$; $\alpha,\beta>0$ such that $\alpha+\beta < \frac{1}{2-p}$;

\While{desired accuracy is not achieved yet}{
  Define $\epsilon_n := [(n+1)^{-\alpha},(n+1)^\beta]$\;

  Calculate $v_{n+1}$ by means of

  \begin{align*}
     \hspace*{-8mm}\int\limits_\Omega (\eps_{n,-} \vee |\nabla v_n| \wedge
     \eps_{n,+})^{p-2}\nabla v_{n+1}\cdot \nabla \xi \d x = \langle
     f,\xi\rangle \qquad \forall \xi \in W_0^{1,2}(\Omega)\;;
  \end{align*}

  Increase $n$ by 1\;


}{
}
\end{algorithm}

\noindent%
We continue to derive a decay estimate for~$\mathcal{G}_n - \mathcal{G}_\infty$.
\begin{lemma}
  \label{lem:decay-G}
  There exists~$K=K(\alpha,\beta,p,q)$ (which appears in the
  definition of~$\mathcal{G}_n$) and
  some~$c_3=c_3(\alpha,\beta,p,q)\geq 1$, such that for
  all~$n \in \setN$
  \begin{align*}
    (\mathcal{G}_{n+1} - \mathcal{G}_\infty) 
    &\leq (1- \tfrac{1}{c_3 (n+1)}) (
      \mathcal{G}_n -
      \mathcal{G}_\infty).
  \end{align*}
\end{lemma}
\begin{proof}
  Define 
  \begin{align*}
    \rho_n := \eps_{-,n}^p + \eps_{+,n}^{-(q-p)} = (n+1)^{-\alpha p} + (n+1)^{-(q-p)\beta}.
  \end{align*}
  Hence it follows by Lemma~\ref{lem:algebraicidff}  in the Appendix that there
  exists~$c_2 = c_2(\alpha,\beta,p,q) \geq 1$ with
  \begin{align}
    \label{eq:5}
    \rho_n - \rho_{n+1} &\geq \frac{1}{c_2} \frac{1}{n+1} \rho_n.
  \end{align}
  In particular,~$\rho_n$ satisfies a decay estimate!

  On the other hand it follows from Theorem~\ref{thm:convalg2} that
  \begin{align*}
    \lefteqn{\mathcal{J}_{\epsilon_n}(v_n) -
    \mathcal{J}_{\epsilon_{n+1}}(v_{n+1})  }\quad 
    &
    \\
    &\geq
      \mathcal{J}_{\epsilon_n}(v_n)
      - \mathcal{J}_{\epsilon_n}(v_{n+1})
    \\
    &\geq \delta_n \,(\mathcal{J}_{\epsilon_n}(v_n)-\mathcal{J}_{\epsilon_n}(u_{\eps_n}))
    \\
    &\geq \frac{1}{c_K} \frac{1}{n+1}
      \,(\mathcal{J}_{\epsilon_n}(v_n)-\mathcal{J}_{\epsilon_n}(u_{\eps_n})) 
    \\
    &= \frac{1}{c_K} \frac{1}{n+1} 
      \,(\mathcal{J}_{\epsilon_n}(v_n)-\mathcal{J}(u)) -
      \frac{1}{c_K} \frac{1}{n+1}
      (\mathcal{J}_{\epsilon_n}(u_{\eps_n}) - \mathcal{J}(u)).
  \end{align*}
  We deduce from~\eqref{eq:RelaxError}, the definition of~$\epsilon_n$
  and~$\rho_n$ that
  \begin{align*}
    \Jeps(u_\eps)- \calJ(u) \le c_R(\epsilon_{n,-}^p +
    \epsilon_{n,+}^{-(q-p)}) 
    = c_R \rho_n.
  \end{align*}
  This and the previous estimate prove
  \begin{align*}
    \mathcal{J}_{\epsilon_n}(v_n) -
    \mathcal{J}_{\epsilon_{n+1}}(v_{n+1})
    &\geq \frac{1}{c_K} \frac{1}{n+1} 
      \,(\mathcal{J}_{\epsilon_n}(v_n)-\mathcal{J}(u)) -
      \frac{c_R}{c_K} \frac{1}{n+1} \rho_n.
  \end{align*}
  Since~$\mathcal{G}_n = \mathcal{J}_{\epsilon_n}(v_n) + K_1\rho_n$,
  it follows together with~\eqref{eq:5} that
  \begin{align*}
    \mathcal{G}_{n+1} - \mathcal{G}_n
    &\geq 
      \frac{1}{c_K} \frac{1}{n+1} 
      \,(\mathcal{J}_{\epsilon_n}(v_n)-\mathcal{J}(u)) + \bigg( \frac{K_1}{c_2} -
      \frac{c_R}{c_K}\bigg) \frac{1}{n+1} \rho_n.
  \end{align*}
  We finally fix~$K_1$: We choose~$K_1$ so large such that
  \begin{align*}
    \frac{K_1}{c_2} -
      \frac{c_R}{c_K} &\geq \frac{K_1}{2\,\max \set{c_K,c_2}},
  \end{align*}
  which is always possible. Combining this with our previous estimates
  we deduce
  \begin{align*}
    \mathcal{G}_{n+1} - \mathcal{G}_n
    &\geq 
      \frac{1}{2\, \max \set{c_K,c_2} } \frac{1}{n+1} (\mathcal{G}_n - \mathcal{G}_\infty).
  \end{align*}
  This proves the theorem with $c_3 = 2\, \max \set{c_K,c_2}$. \qed
\end{proof}
We are now able to present our convergence result.
\begin{theorem}\label{cor:algebraic}
  Let $\nabla u\in L^{q,\infty}(\Omega)$ for some $q>p$ (as given for
  example in Theorem~\ref{thm:regCM} or Theorem \ref{thm:regE}). Then,
  the sequence $(v_n)_{n\in \mathbb N}$ produced by the algorithm
  above described satisfies
  \begin{align*}
    \mathcal{J}_{\epsilon_n}(v_n) - \mathcal{J}(u) \leq
    \calG_n-\calG_\infty \le n^{-\frac{1}{c_3}}\,(\calG_0 -
    \calG_\infty), 
  \end{align*}
  where $c_3$ is the constant of Lemma~\ref{lem:decay-G}. In
  particular, the energy error decreases at least algebraically.
\end{theorem}
\begin{proof}
  The estimate
  $\mathcal{J}_{\epsilon_n}(v_n) - \mathcal{J}(u) \leq
  \calG_n-\calG_\infty$
  is obvious, so it remains to prove the decay
  of~$\mathcal{G}_n - \mathcal{G}_\infty$.  If follows from
  Lemma~\ref{lem:decay-G} that, for~$n \in \setN$,
  \begin{align*}
    \mathcal{G}_n - \mathcal{G}_\infty 
    &\leq 
      \prod_{i=0}^{n-1} \big( 1
      - \tfrac{1}{c_3
      (i+1)}\big)
      (\mathcal{G}_0 - \mathcal{G}_\infty).
  \end{align*}
  Now,
  \begin{align*}
    \prod_{i=0}^n \Big(1- \tfrac{1}{c_3(i+1)} \Big) 
    &= \exp\bigg(
      \sum_{i=0}^n\ln\Big(1- \tfrac{1}{c_3(i+1)} \Big) \bigg) \le \exp\bigg(
      -\sum_{i=0}^n \frac{1}{c_3(i+1)} \bigg)
    \\
    &\le \exp\bigg( - \frac{\ln(n)}{c_3} \bigg) =
      n^{-\frac{1}{c_3}}.
  \end{align*}
  This proves the lemma. \qed
\end{proof}

\begin{remark}
  We have seen that the
  choice~$\epsilon_n = [(n+1)^{-\alpha}, (n+1)^\beta]$ ensures that
  the error decreases at least with an algebraic rate. However, the
  decay of the relaxed energy
  error~$\mathcal{G}_n- \mathcal{G}_\infty$ can never be faster than
  algebraical with this choice of~$\epsilon_n$. Hence, this choice is
  also very restrictive.  From the numerical experiments we performed,
  we have seen that it is possible to decrease~$\epsilon_{n,-}$ and
  increase~$\epsilon_{n,+}$ much faster and still obtain
  convergence. Moreover, the observed convergence is much faster than
  algebraic and more of exponential type. We will present the details
  of such numerical experiments in a subsequent work. Let us
  summarize: the algorithm of this section ensures an algebraic
  convergence rate, but in practice we expect a better behavior for
  other, perhaps adaptive, choices of~$\epsilon_n$, still to be fully
  investigated.
\end{remark}

\section{Numerical Experiments}
\label{sec:numer-exper}

We have performed numerous experiments on the basis of the adaptive
finite element method with piecewise linear elements.  We developed
preliminary versions of error estimators that capture the
effect of the truncation, the adaptivity of the mesh and the
fixpoint iteration. Let $v_n$ denote the iterated solution generated
by the algorithm, then we used the following ad hoc estimators:
\begin{itemize}
\item We use
  \begin{align*}
    \eta^2_{\epsilon^+}(v_n) &:= \mathcal{J}_{\epsilon_n}(v_n) -
                               \mathcal{J}_{(\epsilon_{n,-},\infty)}(v_n) 
  \end{align*}
  to measure the effect of the upper truncation bound~$\epsilon_{n,+}$ and
  \begin{align*}
    \eta^2_{\epsilon^-}(v_n) &:=  \mathcal{J}_{\epsilon_n}(v_n) -
                               \mathcal{J}_{(0,\epsilon_{n,+})}(v_n) 
  \end{align*}
  to measure the effect of the lower truncation bound~$\epsilon_{n,-}$.
\item We use the optimal estimators of \cite{DieK08,BelDieKre11} with
  the Orlicz function~$\phi_{\epsilon_n}$ to estimate the error due to
  mesh refinement, i.e. on elements~$T$ we use the estimators
  \begin{align*}
    \eta^2_h(v_n) &:= \int_T (\phi_{\epsilon_n, \abs{\nabla v_n}})^*(h_T \abs{f})\,dx +
    \sum_{\gamma \subset \partial T} h_\gamma\!
    \int_\gamma \bigabs{\jump{V_{\epsilon_n}(\nabla v_n)}_\gamma}^2\,ds.
  \end{align*}
  
\item To measure the error due to the fixpoint iteration (for~$n \geq 1$)
  \begin{align}\label{eq:vn1eq}
    \eta^2_{\text{Ka{\v c}}}(v_n) &:= \int_\Omega (\phi_{\eps,|\nabla
    v_{n-1}|})^\ast\Big(\tfrac{\phieps'(|\nabla v_{n-1}|)}{|\nabla
    v_{n-1}|}|\nabla(v_{n-1}-v_n)|\Big)\d x, 
  \end{align}
  which is in fact an upper bound for~$\Jeps(v_n)-\Jeps(u_\eps)$ and
  $\Jeps(v_{n-1})-\Jeps(u_\eps)$.
\end{itemize}
We used these estimators to implement a fully adaptive version of our
relaxed $p$-\Kacanov{} iteration.

\begin{algorithm}[H]
\SetAlgoLined
\TitleOfAlgo{Adaptive relaxed $p$-Ka{\v c}anov{} Algorithm}
\KwData{Given~$f \in (W^{1,p}_0(\Omega))^*$, $v_0 \in
  W^{1,2}_0(\Omega)$;}
\KwResult{Approximate solution of the $p$-Poisson
  problem~\eqref{eq:weakppoisson}; }
Initialize: $\eps_0=[1,1] \subset (0,\infty)$,
$n=0$\;

\While{desired accuracy is not achieved yet}{
  Define $a_n := \eps_{n,-} \vee |\nabla v_n| \wedge
  \eps_{n,+}$\;

  Calculate $v_{n+1}$ by means of

  \begin{align*}
     \hspace*{-8mm}\int\limits_\Omega (\eps_{n,-} \vee |\nabla v_n| \wedge
     \eps_{n,+})^{p-2}\nabla v_{n+1}\cdot \nabla \xi \d x = \langle
     f,\xi\rangle \qquad \forall \xi \in W_0^{1,2}(\Omega);\;
  \end{align*}
  Increase total costs by current degrees of freedom\;

  Calculate and compare the error estimators
  $\eta^2_{\epsilon^+}(v_n)$, $\eta^2_{\epsilon^-}(v_n)$,
  $\eta^2_h(v_n)$ and $\eta^2_{\text{Ka{\v c}}}(v_n)$\;

  If $\eta^2_{\epsilon^+}(v_n)$ is the largest, do $\epsilon_{n+1,+}
  := 1.25 \cdot \epsilon_{n,+}$\;

  If $\eta^2_{\epsilon^-}(v_n)$ is the largest, do $\epsilon_{n+1,+}
  := 0.8 \cdot \epsilon_{n,+}$\;

  If $\eta^2_h(v_n)$ is the largest, do a mesh refinement with
  D\"orfler marking\;

  Increase $n$ by 1\;
}{
}
\end{algorithm}

Let us present three experiments that measure different aspects. We
have chosen quite critical situations and small exponents~$p$ in order
to see how the algorithm behaves in particularly bad situations. In
particular, we have chosen a quite small exponent~$p= \frac{16}{15}$
for all of our experiments presented here. (The results of our numerical
experiments behave much nicer for~$p$ closer to~$2$.) The results for
larger exponents are in fact much nicer.  
\begin{itemize}
\item \textbf{The Bump:} On $\Omega = [-1,1]^2$ choose $f$ and the
  boundary values of~$u$ such that
  $u(x)=(x_1^2\!-\!1)(x_2^2\!-\!1)$. There is only one critical point
  at~$(0,0)$, where $\nabla u(0,0)=0$. This example is chosen to see
  how the algorithm behaves for smooth functions with isolated
  extrema. \footnote{Note that $\abs{f(x)}$ behaves like~$\abs{x}^{p-2}$. Thus,
  $f \in L^2$ for all $p>1$ but $f \in L^{p'}$ only for~$p >
  \sqrt{2}$. This makes potential troubles with the used error 
  estimator, but since the error estimator is also truncated
  with~$\epsilon$ the effect is manageable.}
\item \textbf{The Needle:} On $\Omega = [-1,1]^2$ choose $f$ and the
  boundary values of~$u$ such that $u(x) = \abs{x}^{1-\frac 1p}
  -1$. In this case
  $\nabla u(x) = \abs{x}^{-\frac 1p} \frac{x}{\abs{x}}$ and
  $A(\nabla u) = \abs{x}^{-\frac{1}{p'}} \frac{x}{\abs{x}}$ and
  $V(\nabla u) = \abs{x}^{-\frac 12} \frac{x}{\abs{x}}$.  This example
  is chosen such that
  $V(\nabla u) \in W^{\frac 12}L^{2,\infty}(\Omega)$ meaning that half
  a derivative\footnote{This is understood in a heuristic way only:
    half a derivative of~$V(\nabla u)$ growths like the
    function~$\abs{x}^{-1}$, which is in the Lorentz space~$L^{2,\infty}(\Omega)$.}
  of~$V(\nabla u)$ is still in the Lorentz space~$L^{2,\infty}$. This
  corresponds to the regularity of a $p$-harmonic function on a slit
  domain. It is the critical regularity that allows optimal estimates
  for the adaptive finite element method, while for uniform refinement
  a rate of~$O(h^{1/2})$ instead of the optimal one~$O(h)$.
\item \textbf{Constant Force in Slit Domain:} Choose~$f=2$ on the slit
  domain~$\Omega = (-1,1)^2 \setminus [0,1]^2$ and $u=0$ on its
  boundary. There is no explicit formula for the exact solution
  (different from the linear case of~$p=2$) for the solution. The
  reentrant corner reduces the regularity of the solution.
\end{itemize}
There is a nice gap between our theory and the numerical experiments
that we performed. In fact, our experiments shows a significantly faster
convergence rate. This shows that we are on a good track and have
developed a good algorithm.

Figures~\ref{fig:bump}, \ref{fig:needle} and~\ref{fig:f2} show the
results of our experiments. The pictures show a log-log-plot of energy
accuracy vs. the computed costs. Let us explain the picture from
Figure~\ref{fig:bump} in more detail. The others figures are
analogously. The dotted black line is the most important line and
shows the energy difference
$\mathcal{J}_\epsilon(u_n)-\mathcal{J}(u)$, which is the measure of
the error between the computational solution and the exact
solution. The lines $\epsilon_+$ and $\epsilon_-$ represent the
truncation parameters~$(\epsilon_-,\epsilon_+)$.  The other lines
represent error indicators for~$(\epsilon_-,\epsilon_+)$
named~$\eta^2_{\epsilon^+}$ and $\epsilon^2_{\epsilon^-}$, the
fixpoint iteration $\eta^2_{\text{Ka{\v c}}}(u_h)$, and refinement
$\eta_h^2(u_h)$. The straight black line \emph{$\text{costs}^{-1}$} is
the optimal convergence rate (1/costs), where the costs are the
accumulative sum of the degrees of freedom for each step that requires
solving a linear system. (Here we have assumed a linear cost for
solving the linear system, which might be possible with a multi grid
method.) It is important that we use the accumulated cost instead of
the degrees of freedom, since only this truly measures the effort, in
particular if the number of fixpoint iterations increases.

\begin{figure}[ht!]
  \begin{center}
    \begin{tikzpicture}
      \begin{axis}[
        clip=false,
        width=.9\textwidth,
        height=.55\textwidth,
        xmode = log,
        ymode = log,
        xlabel={ndof},
        cycle multi list={\nextlist MyColors},
        scale = {1},
        ymin = {1e-7},
        clip = true,
        legend cell align=left,
        legend style={legend columns=1,legend pos=outer north east,font=\fontsize{7}{5}\selectfont}
        ]
	\addplot table [x=ndof_sum,y=J_eps] {Experiment1.txt};\label{Jeps}
	\addplot table [x=ndof_sum,y=eta_h] {Experiment1.txt};\label{est_refine}
	\addplot table [x=ndof_sum,y=eta_Kac] {Experiment1.txt};\label{est_kacanov}
	\addplot table [x=ndof_sum,y=eta_+] {Experiment1.txt};\label{est_incUp}
	\addplot table [x=ndof_sum,y=eta_-] {Experiment1.txt};\label{est_decLo}
	\addplot table [x=ndof_sum,y=eps_+] {Experiment1.txt};\label{upperShift}
	\addplot table [x=ndof_sum,y=eps_-] {Experiment1.txt};\label{lowerShift}
	\addplot[sharp plot,update limits=false] coordinates {(1e0,1e2) (1e9,1e-7)};\label{invCosts}
	\legend{
          {$\mathcal{J}_\epsilon(u_n)-\mathcal{J}(u)$},
          {$\eta_h^2(u_h)$},
          {$\eta_\textup{Ka\v{c}}^2(u_h)$},
          {$\eta^2_{\varepsilon^+}$},
          {$\eta^2_{\varepsilon^-}$},
          {$\varepsilon_+$},
          {$\varepsilon_-$},
          {$10\,$costs$^{-1}$}};
      \end{axis}
    \end{tikzpicture}
  \end{center}
  \caption{The Bump}
  \label{fig:bump}
\end{figure}

\begin{figure}[ht]
  \begin{center}
    \begin{tikzpicture}
      \begin{axis}[
        clip=false,
        width=.9\textwidth,
        height=.55\textwidth,
        xmode = log,
        ymode = log,
        xlabel={ndof},
        cycle multi list={\nextlist MyColors},
        scale = {1},
        ymin = {1e-7},
        clip = true,
        legend cell align=left,
        legend style={legend columns=1,legend pos=outer north east,font=\fontsize{7}{5}\selectfont}
        ]
	\addplot table [x=ndof_sum,y=J_eps] {Experiment2.txt}; 
	\addplot table [x=ndof_sum,y=eta_h] {Experiment2.txt};
	\addplot table [x=ndof_sum,y=eta_Kac] {Experiment2.txt};
	\addplot table [x=ndof_sum,y=eta_+] {Experiment2.txt};
	\addplot table [x=ndof_sum,y=eta_-] {Experiment2.txt};
	\addplot table [x=ndof_sum,y=eps_+] {Experiment2.txt};
	\addplot table [x=ndof_sum,y=eps_-] {Experiment2.txt};\addplot[sharp plot,update limits=false] coordinates {(1e0,1e3) (1e9,1e-6)};
	\legend{
          {$\mathcal{J}_\epsilon(u_n)-\mathcal{J}(u)$},
          {$\eta_h^2(u_h)$},
          {$\eta_\textup{Ka\v{c}}^2(u_h)$},
          {$\eta^2_{\varepsilon^+}$},
          {$\eta^2_{\varepsilon^-}$},
          {$\varepsilon_+$},
          {$\varepsilon_-$},
          {$100\,$costs$^{-1}$}};
      \end{axis}
    \end{tikzpicture}
  \end{center}
  \caption{The Needle}
  \label{fig:needle}
\end{figure}

\begin{figure}[ht]
  \begin{center}
    \begin{tikzpicture}
      \begin{axis}[
        clip=false,
        width=.9\textwidth,
        height=.55\textwidth,
        xmode = log,
        ymode = log,
        xlabel={ndof},
        cycle multi list={\nextlist MyColors},
        scale = {1},
        ymin = {1e-7},
        clip = true,
        legend cell align=left,
        legend style={legend columns=1,legend pos=outer north east,font=\fontsize{7}{5}\selectfont}
        ]
	\addplot table [x=ndof_sum,y=J_eps] {Experiment3.txt};
	\addplot table [x=ndof_sum,y=eta_h] {Experiment3.txt};
	\addplot table [x=ndof_sum,y=eta_Kac] {Experiment3.txt};
	\addplot table [x=ndof_sum,y=eta_+] {Experiment3.txt};
	\addplot table [x=ndof_sum,y=eta_-] {Experiment3.txt};
	\addplot table [x=ndof_sum,y=eps_+] {Experiment3.txt};
	\addplot table [x=ndof_sum,y=eps_-] {Experiment3.txt};
	\addplot[sharp plot,update limits=false] coordinates {(1e0,1e2) (1e9,1e-7)};
	\legend{
          {$\mathcal{J}_\epsilon(u_n)-\mathcal{J}(u)$},
          {$\eta_h^2(u_h)$},
          {$\eta_\textup{Ka\v{c}}^2(u_h)$},
          {$\eta^2_{\varepsilon^+}$},
          {$\eta^2_{\varepsilon^-}$},
          {$\varepsilon_+$},
          {$\varepsilon_-$},
          {$10\,$costs$^{-1}$}};
      \end{axis}
    \end{tikzpicture}
  \end{center}
  \caption{Constant Force in Slit Domain}
  \label{fig:f2}
\end{figure}

Let us explain the numerical results in more detail.
\begin{itemize}
\item \textbf{Bump:} The algorithm converges optimally with respect to
  the accumulated costs. The exact solution has a bounded gradient, so
  the upper truncation bound~$\epsilon_{n,+}$ does not increase
  much. There is only one critical point at~$x=0$, where
  $\nabla u(0)=0$. Thus, the algorithm decreases the lower truncation
  bound~$\epsilon_{n,-}$ quite moderately, since the error from
  truncation appears only in a small neighborhood around zero. All
  estimators decrease nicely with the optimal rate (1/cost). The
  number of fixpoint iterations remains bounded, so the cost is
  proportional to the current number of degrees of freedom.
\item \textbf{The Needle:} The algorithm converges optimally with respect to
  the accumulated costs. The exact solution has unbounded gradient at
  the isolated point zero. Therefore, the upper truncation
  bound~$\epsilon_{n,+}$ is increased by the algorithm. This upper
  truncation starts
  quite late, since it is only necessary in set of small measure. The
  number of fixpoint iterations remain bounded.
\item \textbf{Constant Force in Slit Domain:} The gradient of the
  exact solution is bounded, so the upper truncation
  bound~$\epsilon_{n,+}$ does not increase much. The energy error
  still decreases very fast. It decreases almost optimally with
  respect to the accumulated cost, but the slope seems slightly worse.
  It is however still much faster, than our worst-case theory
  predicts. 
\end{itemize}
Overall, we see that our algorithm converges with a
rate, which is optimal in many cases.

\appendix
\section{Auxiliaries}
\label{sec:auxiliaries}

The following embedding is probably well known. However, since we
could not find a reference for it and we need it for the proof of
Theorem~\ref{thm:regE}, we include a short proof of it.\footnote{We thank
  W.~Sickel for explaining the details to us.}  In the
following~$\mathcal{N}^{\frac 12,2}$ denotes the usual \Nikolskii{}
space, see e.g.~\cite{KOF77}.
\begin{lemma}\label{lem:embedding}
  $\calN^{\frac 12,2}(\Omega)\hookrightarrow
  L^{\frac{2d}{d-1},\infty}(\Omega)$.
\end{lemma}
\begin{proof}
  We will not recapitulate the definitions of the occurring
  spaces. First of all, we use the identity
  \begin{align*}
    \mathcal{N}^{\frac 12,2}(\Omega) = B_{2,\infty}^{\frac 12}(\Omega)
  \end{align*}
  as stated in \cite[Remark 8.4.5]{KOF77}, where $B_{p,q}^s(\Omega)$
  denotes the standard Besov spaces. In \cite[Theorems 1 and 2 in
  4.3.1]{T78} we find the interpolation pair $\{B^{\frac{1}{4}}_{2,1}
  (\Omega), B^{\frac{3}{4}}_{2,1} (\Omega)\}$ such that
  \begin{align*}
    B_{2,\infty}^{\frac 12}(\Omega) = (B^{\frac{1}{4}}_{2,1} (\Omega),
    B^{\frac{3}{4}}_{2,1} (\Omega))_{\frac{1}{2}, \infty} 
  \end{align*}
  holds. The embeddings (see \cite{EEK06} respectively \cite{Pee66})
  \begin{align*}
    B^{\frac{1}{4}}_{2,1} (\Omega)\hookrightarrow
    L^{\frac{4d}{2d-1}}(\Omega) \quad\text{and}\quad
    B^{\frac{3}{4}}_{2,1} (\Omega)\hookrightarrow
    L^{\frac{4d}{2d-3}}(\Omega)  
  \end{align*}
  yield
  \begin{align*}
    (B^{\frac{1}{4}}_{2,1} (\Omega), B^{\frac{3}{4}}_{2,1}
    (\Omega))_{\frac{1}{2}, \infty}\hookrightarrow
    (L^{\frac{4d}{2d-1}}(\Omega),
    L^{\frac{4d}{2d-3}}(\Omega))_{\frac{1}{2}, \infty} . 
  \end{align*}
  Finally, by \cite[Theorem 2 in 1.18.6]{T78} we get
  \begin{align*}
    (L^{\frac{4d}{2d-1}}(\Omega),
    L^{\frac{4d}{2d-3}}(\Omega))_{\frac{1}{2}, \infty} &=
                                                         L^{\frac{2d}{d-1},\infty}(\Omega)
                                                         . 
  \end{align*}
  This proves the claim. \qed
\end{proof}

Moreover, in the proof Lemma~\ref{lem:decay-G} we used the
following algebraic estimate:
\begin{lemma}
  \label{lem:algebraicidff}%
  Let $\gamma>0$. Then for all $n\ge 1$ we have
  \begin{align*}
    n^{-\gamma}  -(n+1)^{-\gamma} \ge n^{-\gamma-1}\min\{\tfrac\gamma 2,1-2^{-\gamma}\} .
  \end{align*}
\end{lemma}
\begin{proof}
  We define $h:[0,\tfrac 12]\to\setR$ via $h(t):=1-(1-t)^\gamma$. Note
  that $h'(t) = \gamma(1-t)^{\gamma-1}$ and
  $h''(t) = \gamma(1-\gamma)(1-t)^{\gamma-2}$. For $\gamma\ge 1$ this
  implies that $h$ is concave, so
  \begin{align*}
    h(t)&\ge t\Big(\tfrac{h(\frac 12)-h(0)}{\frac 12}\Big)= 2 (1-2^{-\gamma})t .
  \end{align*}
  On the other hand, if $\gamma\in(0,1)$, the function $h$ is convex. Therefore,
  \begin{align*}
    h(t)&\ge h(0) + th'(0) = \gamma t .
  \end{align*}
  This implies $h(t)\ge \min\{\gamma, 2(1-2^{-\gamma})\}t$. Therefore, we get
  \begin{align*}
    n^{-\gamma} - (n+1)^{-\gamma}
    &=
      n^{-\gamma}(1-(1-\tfrac{1}{n+1})^\gamma)
      = n^{-\gamma}h(\tfrac{1}{n+1}) 
    \\
    &\ge n^{-\gamma}\min\{\gamma, 2(1-2^{-\gamma})\}\tfrac{1}{n+1}\ge
      n^{-\gamma-1}\min\{\tfrac\gamma 2,1-2^{-\gamma}\} . 
  \end{align*}
  This proves the claim. \qed
\end{proof}

\section{On Uniformly Convex Orlicz Functions}
\label{sec:orlicz-functions}

In this appendix we introduce the concept of \emph{uniformly convex
  Orlicz functions} and their shifted versions. The results presented
below are modifications of \cite[Lemma~3]{DE08},
\cite[Lemma~16]{DieK08}, and~\cite[Corollary~26]{DieK08}. However,
since we use here a slightly different notion of shifted functions and
less regularity for our Orlicz functions, we decided to include a
proof and keep our paper self-contained. Throughout this section we
assume that our Orlicz function satisfies the following assumptions.
\begin{definition}
  \label{def:general-phi}
  Let $\phi$ be an N-function\footnote{See the beginning of
    Section~\ref{sec:epslimit} for the definition of an N-function.}.
  We say that $\phi$ is \emph{uniformly convex} if there
  exist~$c_4,c_5 >0$ with
  \begin{align}
    \label{eq:mon}
    c_4 \frac{\phi'(s)}{s} 
    &\leq \frac{\phi'(s) - \phi'(t)}{s-t} \leq 
      c_5 \frac{\phi'(s)}{s} \qquad \text{for all $s>t \geq 0$}.
  \end{align}
\end{definition}
The choice~$t=0$ implies that $c_4 \leq 1 \leq c_5$. Moreover,~\eqref{eq:mon} implies that~$\phi'$ is strict increasing.

For any N-function~$\rho$ the complementary (or dual) N-function
$\rho^*$ is given by $\rho^*(u) = \sup_{t \geq 0} (ut -
\phi(t))$. This is equivalent to $(\rho^*)' = (\rho')^{-1}$. Note
that~$(\rho^*)^*=\rho$. 
\begin{lemma}
  \label{lem:mon-is-dual}
  If $\phi$ is a uniformly convex N-function with constants~$c_4$
  and~$c_5$, then $\phi^*$ is a uniformly convex N-function with
  constants $1/c_5$ and $1/c_4$.
\end{lemma}
\begin{proof}
  Let $s=(\phi^*)'(u)$ and $t=(\phi^*)'(v)$. Then by strict
  monotonicity of~$\phi'$ and $(\phi^*)'$ condition~\eqref{eq:mon} is
  equivalent to
  \begin{align*}
    c_4 \frac{u}{(\phi^*)'(u)}
    &\leq \frac{u-v}{(\phi^*)'(u) - (\phi^*)'(v)} \leq 
      c_5 \frac{u}{(\phi^*)'(u)} \qquad \text{for all $u>v \geq 0$.}
  \end{align*}
  Taking the reciprocal proves the uniform convexity of~$\phi^*$. The
  reverse conclusion follows by duality, i.e.~$\phi =(\phi^*)^*$. \qed
\end{proof}

\begin{lemma}
  Let $\phi$ be an N-function, which is piecwise~$C^2$ on
  $(0,\infty)$. Then~$\phi$ is uniformly convex if and only if
  \begin{align}
    \label{eq:old-mon}
    c_6 \leq \frac{\phi''(t)\,t}{\phi'(t)} 
    &\leq c_7 \qquad 
      \text{for
      all $t>0$} 
  \end{align}
  for some $c_6,c_7>0$.  This is in fact the uniform convexity
  condition in~\cite{DE08}.
\end{lemma}
\begin{proof}
  Note that $s \searrow t$ in~\eqref{eq:mon}
  implies~\eqref{eq:old-mon} with $c_6=c_4$ and $c_7=c_5$. So assume
  now, that~\eqref{eq:old-mon} holds.

  If $s \geq 2t$, then $s-t \geq s/2$ and the upper bound
  of~\eqref{eq:mon} is obvious with $c_7=2$.
  So let us assume $t < s \leq 2t$. Then
  \begin{align*}
    \frac{\phi'(s)-\phi'(t)}{s-t}
    &= \frac{1}{s-t}\int_t^s \phi''(\tau)\,d\tau
      \leq \frac{c_7}{s-t} \int_t^s \frac{\phi'(\tau)}{\tau}\,d\tau
      \leq \frac{c_7 \phi'(s)}{s-t} \log(s/t).
  \end{align*}
  Now $\log(1+a) \leq a$ for~$a\geq 0$ and $s \leq 2t$ imply
  \begin{align*}
    \frac{\phi'(s)-\phi'(t)}{s-t}
    &\leq \frac{c_7 \phi'(s)}{s-t} \frac{s-t}{t} = \frac{c_7
      \phi'(s)}{t} \leq 2c_7 \frac{\phi'(s)}{s}. 
  \end{align*}
  This proves the upper bound of~\eqref{eq:mon}. The lower bound
  follows by duality: Indeed, it follows from
  \begin{align*}
    \frac{(\phi^*)''(u)u}{(\phi^*)'(u)}
    &=
      \frac{\phi'(s)}{\phi''(s)\,s}
      \qquad \text{with $\phi'(s)=u$}
  \end{align*}
  that $\phi^*$ satisfies~\ref{eq:old-mon} with constants~$1/c_7$ and
  $1/c_6$. Thus, by the already proven $\phi^*$ satisfies the upper
  estimate of~\eqref{eq:mon}. By duality, see
  Lemma~\ref{lem:mon-is-dual}.  we get the lower estimate of~$\phi$
  in~\eqref{eq:mon}. \qed
\end{proof}
\begin{lemma}
  \label{lem:uniform-convex-delta2}
  Let $\phi$ be a uniformly convex N-function.  Then the functions~$\phi$
  and~$\phi^*$ satisfy the~$\Delta_2$-condition, with constants only
  depending on~$c_4$ and $c_5$.
\end{lemma}
\begin{proof}
  Using $s=\lambda t$ with~$\lambda >1$ in~\eqref{eq:mon} we obtain
  \begin{align}
    \label{eq:phip-delta2}
    \Big( 1 - c_5 \frac{\lambda-1}{\lambda}\Big) \phi'(\lambda t) \leq \phi'(t).
  \end{align}
  Now, we can choose~$\lambda_0 > 1$ such that
  $\mu := (1-c_4 \frac{\lambda_0-1}{\lambda_0}) > \frac 12$. So, we
  have $\phi'(\lambda_0 t) \leq 2 \phi'(t)$. From this it
  follows by iteration (also using the monotonicity of~$\phi')$ that
  $\phi'(2t) \leq \tilde{c}\, \phi'(t)$, where~$\tilde{c}$ only depends
  on~$\lambda_0$ and~$\mu$ and therefore only on~$c_4$.  Thus, it
  follows that
  \begin{align*}
    \phi(2t)-\phi(t) 
    &= 2\,\int_{t/2}^{t} \phi'(2s)\,ds
      \leq 2\,\tilde{c} \int_{t/2}^{t} \phi'(s)\,ds 
      = 2\,\tilde{c}\,
      \big(\phi(t) - \phi(t/2) \big) \leq 2 \tilde{c}\, \phi(t).
  \end{align*}
  Hence, $\phi(2t) \leq (1+2\tilde{c}) \phi(t)$, which proves
  the claim for~$\phi$. The claim for~$\phi^*$ follows by duality with
  Lemma~\ref{lem:mon-is-dual}. \qed
\end{proof}
\begin{remark}[Young's inequality]
  \label{rem:young}
  Let $\phi$ be an N-function such that $\phi$ and $\phi^*$
  satisfy the $\Delta_2$-condition. Then for every $s,t \geq 0$
  we have by \emph{Young's inequality}
  \begin{align*}
    st &\leq \phi(s) +\phi^*(t).
  \end{align*}
  This and $\phi(t) \leq \phi'(t)\,t \leq \phi(2t)$ implies that for
  all~$s,t\geq 0$
  \begin{align*}
    \phi'(s)\,t &\leq \delta \phi(s) + c_\delta \phi(t),
    \\
    \phi'(s)\,t &\leq c_\delta \phi(s) + \delta \phi(t),
  \end{align*}
  where $c_\delta$ depends only on~$\delta$ and the $\Delta_2$-constants.
\end{remark}

For each~$a\geq 0$ we define the \emph{shifted} N-function $\phi_a$ by
its derivative
\begin{align}
  \label{eq:shifted-phi}
  \phi_a'(t) &:= \frac{\phi'(t \vee a)}{t \vee a} t
\end{align}
and $\phi_a(t) = \int_0^t \phi_a'(\tau)\,d\tau$.  In the notation of
Section~\ref{sec:epslimit} this is just $\phi_\epsilon$ with
$\epsilon=(a,\infty)$, see also Remark~\ref{rem:no-upper-bound}.
\begin{lemma}
  \label{lem:shift_conj_exact}
  There holds $(\phi_a)^* = (\phi^*)_{\phi'(a)}$.
\end{lemma}
\begin{proof}
  Note that
  $((\phi^*)_{\phi'(a)})'(u) = \frac{(\phi^*)'(u \vee \phi'(a))}{u \vee
    \phi'(a)} u$ is the inverse of $\phi_a'(t)$. Thus
  $(\phi^*)_{\phi'(a)}$ and $(\phi_a)^*$ are conjugate to each other. \qed
\end{proof}

\begin{remark}
  The shifted N-functions have already been originally introduced
  in~\cite{DE08} with the modified definition
  $\phi_a'(t) = \frac{\phi'(t + a)}{t+a}\,t$. This original version
  shares almost all of the properties with the version of this
  paper. However, our exact formula $(\phi_a)^* = (\phi^*)_{\phi'(a)}$
  of Lemma~\ref{lem:shift_conj_exact} is replaced in~\cite{DE08} by
  equivalence. This is one of the advantages of our new definition.
\end{remark}

Given our uniformly convex N-function~$\phi$ we define the stress~$A$
and the auxiliary function~$V$ as in Definition~\ref{def:AandV} by
\begin{definition}
  \label{def:A-and-V}
  For $P\in\setR^d$ we define
  \begin{align*}
    A(P):= 
    \begin{cases}
      \frac{\phi'(\vert P\vert)}{\vert P\vert}P &\text{if } P\neq 0
      \\
      0&\text{if }P=0\end{cases}\quad\text{and}\quad
    V(P):=\begin{cases} \sqrt{\frac{\phi'(\vert P\vert)}{\vert
          P\vert}}P&\text{if }P\ne 0\\0&\text{if }P=0.
    \end{cases}
  \end{align*}
\end{definition}

\begin{lemma}
  \label{lem:psi}
  Let $\phi$ be a uniformly convex N-function with constants~$c_4$
  and~$c_5$. Define $\psi$ via its derivative by
  $\frac{\psi'(s)}{s} := \sqrt{\frac{\phi'(s)}{s}}$ and let
  $\psi(t) := \int_0^t \psi'(s)\,ds$. Then $\psi$ is a uniformly
  convex N-function with constants $\frac 12$ and $1+c_5$. Moreover,
  $V(P) = \frac{\psi'(\abs{P})}{\abs{P}} P$.
\end{lemma}
\begin{proof}
  For $s>t$ we calculate
  \begin{align*}
    I &:= \frac{\psi'(s)- \psi'(t)}{s-t} 
        = 
        \frac{\sqrt{\phi'(s)s}- \sqrt{\phi'(t) t}}{s-t}
        =
        \frac{1}{\sqrt{\phi'(s)s}+ \sqrt{\phi'(t) t}}
        \frac{\phi'(s) s- \phi'(t) t}{s-t}
    \\
      &=
        \frac{\phi'(s)}{\sqrt{\phi'(s)s}+ \sqrt{\phi'(t) t}} +
        \frac{\phi'(s) - \phi'(t)}{s-t}   \frac{t}{\sqrt{\phi'(s)s}+
        \sqrt{\phi'(t)  t}}
      =: II + III.
  \end{align*}
  Now, with~\eqref{eq:mon}
  \begin{align*}
    \frac 12 \frac{\psi'(s)}{s} 
    &= \frac{\sqrt{\phi'(s)}}{2 \sqrt{s}}
      \leq II \leq  \frac{\sqrt{\phi'(s)}}{\sqrt{s}} =
      \frac{\psi'(s)}{s}
  \end{align*}
  and
  \begin{align*}
    0 &\leq III \leq c_5 \frac{\phi'(s)}{s} \frac{t}{\sqrt{\phi'(s)s}+
        \sqrt{\phi'(t)  t}}
        \leq c_5 \frac{\sqrt{\phi'(s)}}{\sqrt{s}} = c_5 \frac{\psi'(s)}{s}. 
  \end{align*}
  Overall, we get $\frac 12 \frac{\psi'(s)}{s} \leq I \leq (1+c_5)
  \frac{\psi'(s)}{s}$, which proves the claim. \qed
\end{proof}

\begin{lemma}
  \label{lem:shifted-uniform}
  Let $\phi$ be a uniformly convex N-function. Then $\phi_a$
  is also uniformly convex with constants~$c_4$ and $c_5$ replaced by
  $c_4$ and $c_5+1$. 
\end{lemma}
\begin{proof}
  For $s>t$ define
  \begin{align*}
    I &:= \frac{\phi_a(s)-\phi_a(t)}{s-t} \qquad \text{and} \qquad II
        := \frac{\phi_a'(s)}{s}.
  \end{align*}
  If $s,t \geq a$, then $\phi_a'(s)=\phi'(s)$ and
  $\phi_a'(t)=\phi'(t)$. Hence $c_4 II \leq I \leq c_5 II$ by
  assumptions on~$\phi$.

  If $s,t \leq a$, then $\phi_a(s) = \frac{\phi'(a)}{a}
  s$ and $\phi_a(t) = \frac{\phi'(a)}{a}
  t$, so $I= \frac{\phi'(a)}{a} =II$. Since $c_4 \leq 1 \leq c_5$ the claim follows in
  this case.

  If remains to consider $s > a > t$. Then
  $\phi_a'(s)=\phi'(s)$ and $\phi_a'(t) =
  \frac{\phi'(a)}{a} s$, so
  \begin{align*}
    I &= \frac{\phi'(s) - \frac{\phi'(a)}{a} t}{s-t} =
        \frac{\phi'(s) - \phi'(a)}{s-t} + 
        \frac{\phi'(a)}{a}  \frac{a-t}{s-t} =:
        I_1 + I_2.
  \end{align*}
  Thus,
  \begin{align*}
    I &\geq         \frac{\phi'(s) - \phi'(a)}{s-a} + 0
        \geq c_4 \frac{\phi'(s)}{s}.
  \end{align*}
  On the hand using that $a \mapsto
  \frac{a-t}{a}$ is increasing in~$a$ we get
  \begin{align*}
    I &\leq         \frac{\phi'(s) - \phi'(t)}{s-a} +
        \frac{\phi'(a)}{s}  \frac{s-t}{s-t} 
        \leq c_5 \frac{\phi'(s)}{s} + \frac{\phi'(s)}{s}.
  \end{align*}
  This proves the claim. \qed
\end{proof}
\begin{remark}
  If follows from Lemma~\ref{lem:shifted-uniform} and
  \ref{lem:uniform-convex-delta2} that the
  families~$\set{\phi_a}_{a \geq 0}$ and $\set{(\phi_a)^*}_{a \geq 0}$
  satisfy the~$\Delta_2$-condition with constants independent of~$a$.
\end{remark}

\begin{lemma}
  \label{lem:mon}
  Let $\phi$ be a uniformly convex N-function with constants~$c_4$
  and~$c_5$. Then for all~$P,Q$ we have
  \begin{align*}
    (c_4 \wedge \tfrac 12)  \frac{\phi'(\abs{P} \vee
    \abs{Q})}{\abs{P} \vee \abs{Q}} \abs{P-Q}^2
    &\leq (A(P)- A(Q)) : (P-Q) 
    \\
    &\leq (c_5 \vee 2) \frac{\phi'(\abs{P} \vee
      \abs{Q})}{\abs{P} \vee \abs{Q}} \abs{P-Q}^2,
    \\
    \abs{A(P) - A(Q)} &\leq (c_5 + 1) \frac{\phi'(\abs{P} \vee
                       \abs{Q})}{\abs{P} \vee \abs{Q}} \abs{P-Q}.
  \end{align*}
\end{lemma}
\begin{proof}
  We define $\hat{Q} := \frac{P}{\abs{P}}$,
  $\hat{P} := \frac{P}{\abs{P}}$, $\theta := \hat{P} : \hat{Q}$,
  \begin{align*}
    f(P,Q) &:= (A(P)- A(Q)) : (P-Q)
     \quad \text{and} \quad
    g(P,Q) 
           := \frac{\phi'(\abs{P} \vee \abs{Q})}{\abs{P}\vee \abs{Q}}
             \abs{P-Q}^2. 
  \end{align*}
  Then
  \begin{align*}
    f(P,Q) 
    &= \big( \phi'(\abs{P}) \hat{P} - \phi'(\abs{P}) \hat{Q}\big)
      :(\abs{P} \hat{P} - \abs{Q} \hat{Q})
    \\
    &=  \phi'(\abs{P})\abs{P} + \phi'(\abs{Q}) \abs{Q} -\big(
      \phi'(\abs{P})\abs{Q} + \phi'(\abs{Q}) \abs{P}\big) \theta
    =: f(\abs{P},\abs{Q}, \theta)
  \end{align*}
  and
  \begin{align*}
    g(P,Q) 
    &= \frac{\phi'(\abs{P} \vee \abs{Q})}{\abs{P}\vee \abs{Q}} \big(
      \abs{P}^2 + \abs{Q}^2 - 2 \theta \abs{P} \abs{Q}\big)
    =: g(\abs{P},\abs{Q}, \theta).
  \end{align*}
  We need to estimate
  $f(\abs{P},\abs{Q},\theta)/g(\abs{P},\abs{Q},\theta)$. Both~$f$
  and~$g$ are non-negative and linear in~$\theta$. Hence~$\frac{f}{g}$
  is monotone in~$\theta$. In particular, it suffices to control the
  cases~$\theta=1$ and $\theta=-1$.
  We begin with the simple case~$\theta=-1$.
  \begin{align*}
    \frac{f(\abs{P},\abs{Q},-1)}{g(\abs{P},\abs{Q},-1)} 
    &=
      \frac{(\phi'(\abs{P}) +
      \phi'(\abs{Q}))(\abs{P}+\abs{Q})}{ (\phi'(\abs{P}) \vee
      \phi'(\abs{Q}))(\abs{P} +\abs{Q})}.
  \end{align*}
  Since $a \vee b \leq a+b \leq 2 (a \vee b)$, this implies
  \begin{align*}
    \frac 12 &\leq \frac{f(\abs{P},\abs{Q},-1)}{g(\abs{P},\abs{Q},-1)}
               \leq 2.
  \end{align*}
  Now consider the case~$\theta=1$. Without loss of generality we can
  assume $\abs{P} \geq \abs{Q}$.
  \begin{align*}
    \frac{f(\abs{P},\abs{Q},1)}{g(\abs{P},\abs{Q},1)} 
    &=
      \frac{(\phi'(\abs{P}) -
      \phi'(\abs{Q}))(\abs{P}-\abs{Q}) (\abs{P} \vee \abs{Q})}{
      (\phi'(\abs{P}\vee \phi'(\abs{Q})))\,(\abs{P} 
      -\abs{Q})^2}= 
      \frac{(\phi'(\abs{P}) -
      \phi'(\abs{Q}))\abs{P}}{ \phi'(\abs{P})\,(\abs{P}
      -\abs{Q})}. 
  \end{align*}
  Now,~\eqref{eq:mon} implies
  \begin{align*}
    c_4 &\leq \frac{g(\abs{P},\abs{Q},1)}{f(\abs{P},\abs{Q},1)} \leq c_5.
  \end{align*}
  Combining the two cases proves the first claim. To prove the second
  one we assume again~$\abs{P} \geq \abs{Q}$ and estimate
  \begin{align*}
    \abs{A(P) - A(Q)}
    &= \bigabs{\phi'(\abs{P}) \hat{P} - \phi'(\abs{Q}) \hat{Q}}
      \leq \bigabs{\phi'(\abs{P}) - \phi'(\abs{Q})} + \phi'(\abs{Q})
      \abs{\hat{P} - \hat{Q}}
    \\
    &\leq (c_5+1) \frac{\phi'(\abs{P} \vee
      \abs{Q})}{\abs{P} \vee \abs{Q}} \abs{P-Q}.
  \end{align*}
  This proves the second claim. \qed
\end{proof}
\begin{lemma}
  \label{lem:phi-shift-equiv}
  Let $\phi$ be a uniformly convex N-function with constants~$c_4$
  and~$c_5$. Then for all~$P,Q \in \setR^d$
  \begin{align*}
    \phi_{\abs{Q}}'(\abs{P - Q})
    &\eqsim
      \frac{\phi'(|P|\vee
      |Q|)}{|P|\vee |Q|}\abs{P-Q},
    \\
    \phi_{\abs{Q}}(\abs{P - Q})
    &\eqsim
      \frac{\phi'(|P|\vee
      |Q|)}{|P|\vee |Q|}\abs{P-Q}^2,
  \end{align*}
  where the constants only depend on~$c_4$ and $c_5$.
\end{lemma}
\begin{proof}
  The estimates follows
  from~$\frac 12 (\abs{Q} + \abs{P-Q}) \leq \abs{P} \vee \abs{Q} \leq
  2 (\abs{Q} + \abs{P-Q})$, the~$\Delta_2$-property of~$\phi$ (see
  Lemma~\ref{lem:uniform-convex-delta2}) and
  $\phi_{\abs{Q}}'(t)\,t \eqsim \phi_{\abs{Q}}(t)$. \qed
\end{proof}

\begin{lemma}
  \label{lem:hammer}
  Let $\phi$ be a uniformly convex N-function with constants~$c_4$
  and~$c_5$ and let $A$ and~$V$ be as in Lemma~\ref{def:A-and-V}. Then
  for all $P,Q\in\setR^d$
  \begin{align*}
    (A(P)-A(Q))\cdot (P-Q)
    &\eqsim \phi_{\abs{Q}}(\abs{P - Q}) \eqsim \abs{V(P) -
      V(Q)}^2,
    \\
    \abs{A(P)-A(Q)}
    &\eqsim \phi_{\abs{Q}}'(\abs{P - Q}).    
  \end{align*}
  where the constants only depend on~$c_4$ and $c_5$.
\end{lemma}
\begin{proof}
  The first equivalence in the first claim and the second claim follow
  immediately from Lemma~\ref{lem:mon} and~\ref{lem:phi-shift-equiv}.
  It remains to prove the second equivalence of the first claim. For
  this we recall~$\psi$ from Lemma~\ref{lem:psi} and observe that~$V$
  is induced by~$\psi$ exactly as~$A$ is induced by~$\phi$, see
  Definition~\ref{def:A-and-V}. In particular, we
  have~$\abs{V(P)-V(Q)} \eqsim \psi_{\abs{Q}}'(\abs{P-Q})$. This proves
  \begin{align*}
    \abs{V(P) -
    V(Q)}^2 &\eqsim \bigabs{\psi_{\abs{Q}}'(\abs{P-Q})}^2
              =\bigg(\frac{\psi'(\abs{Q} \vee \abs{P-Q})}{\abs{Q} \vee
              \abs{P-Q}}
              \abs{P-Q}\bigg)^2
    \\
            &
              = \frac{\phi'(\abs{Q} \vee \abs{P-Q})}{\abs{Q} \vee
              \abs{P-Q}} \abs{P-Q}^2 \eqsim
              \phi_{\abs{Q}}(\abs{P-Q}).
  \end{align*}
  This proves the claim. \qed
\end{proof}
\begin{lemma}
  \label{lem:JAVpre}
  The following estimates hold for arbitrary
  $v\in W_0^{1,\phi}(\Omega)$ and $u$ being the minimizer of
  $\mathcal{J}(w) := \int_\Omega \phi(\abs{\nabla w}) - fw\,dx$:
  \begin{align*}
    \mathcal{J}(v)-\mathcal{J}(u)&\le \int\limits_\Omega (A(\nabla
    v)-A(\nabla u))\cdot\nabla (v-u)\d x
    \\
    &\eqsim \int\limits_\Omega |V(\nabla v)-V(\nabla
    u)|^2\d x
    \\
    &\lesssim \mathcal{J}(v)-\mathcal{J}(u) .
  \end{align*}
\end{lemma}
\begin{proof}
  It follows by convexity and~$\mathcal{J}'(u)=0$ that
  \begin{align*}
    \mathcal{J}(v) - \mathcal{J}(u) &\leq \mathcal{J}'(v)(v-u) =
                                      (\mathcal{J}'(v)-\mathcal{J}'(u))(v-u). 
  \end{align*}
  Since
  $\mathcal{J}'(w)(\xi) = \int_\Omega A(\nabla w) \nabla \xi -
  f\,\xi\,dx$, this proves the first estimate. The equivalence in the
  claim follows then by Lemma~\ref{lem:hammer}. For the last estimate
  we calculate with Taylor, $\mathcal{J}'(u)=0$,
  Lemma~\ref{lem:hammer} and the uniform~$\Delta_2$-condition of the
  family~$\phi_a$ that
  \begin{align*}
    \mathcal{J}(v) - \mathcal{J}(u)
    &= \int_0^1 \big(\mathcal{J}'(u+t(v-u))-\mathcal{J}'(u)\big)
      (v-u)\,dt
    \\
    &= \int_0^1 \int_\Omega \big(A(\nabla (u+t(v-u))) - A(\nabla u)\big) \cdot \nabla
      (v-u) \,dx \, dt
    \\
    &\eqsim \int_\Omega \int_0^1 \phi_{\abs{\nabla u}}(t \abs{\nabla (v-u)}) 
      \frac{dt}{t} \,dx
      \geq \int_\Omega \int_{\frac 12}^1  \phi_{\abs{\nabla
      u}}\big(\tfrac 12 \abs{\nabla (v-u)}\big)  
      dt \,dx
    \\
    &\eqsim \int_\Omega \phi_{\abs{\nabla u}}\big(\abs{\nabla  v - \nabla u}\big)\,dx
      \eqsim \int_\Omega \abs{V(\nabla v)-V(\nabla u)}^2\,dx.
  \end{align*}
  This proves the claim. \qed
\end{proof}
The following lemma is a sharper version of Lemma~25 and Lemma~27
of~\cite{DieK08}.
\begin{lemma}[Shift-change]
  \label{lem:shift-change-new}
  For all~$a,b \geq 0$ and~$t \geq 0$, there holds
  \begin{align}
    \label{eq:shift-change-new1}
    \abs{\phi_a'(t) - \phi_b'(t)} &\lesssim \phi_a'(\abs{a-b}) \eqsim
                                    \phi_b'(\abs{a-b}),
    \\
    \label{eq:shift-change-new2}
    \bigabs{\big((\phi_a)^*\big)'(t) - \big(\phi_b)^*\big)'(t)}
                                  &\lesssim \abs{a-b}. 
  \end{align}
\end{lemma}
\begin{proof}
  We begin with the proof of~\eqref{eq:shift-change-new1}.  The
  equivalence in~\eqref{eq:shift-change-new1} follows from
  Lemma~\ref{lem:phi-shift-equiv}. Thus the claim is symmetric in~$a$
  and $b$ and we can assume that $a \leq b$.  If $a \leq b \leq t$,
  then $\phi_a'(t)=\phi_b'(t)=0$ and~\eqref{eq:shift-change-new1}
  follows. If $t \leq a \leq b$, then with Lemma~\ref{lem:hammer}
  \begin{align*}
    \bigabs{\phi_a'(t)-\phi_b'(t)}
    &= \Bigabs{\frac{\phi'(a)}{a}t - \frac{\phi'(b)}{b} t}
      \leq \abs{\phi'(a)-\phi'(b)} \frac ta + \phi'(b) \Bigabs{\frac
      1a -\frac 1b} t
    \\
    &\lesssim \phi_a'(\abs{a-b}) \frac t a + \phi'(b)
       \frac{\abs{b-a}\,t}{a b} \lesssim \phi_b'(\abs{a-b}).
  \end{align*}
  This proves~~\eqref{eq:shift-change-new1}.  It remains to consider
  the case~$a \leq t \leq b$. In this situation we estimate using the
  previous two cases
  \begin{align*}
    \bigabs{\phi_a'(t)-\phi_b'(t)}
    &\leq
      \bigabs{\phi_a'(t)-\phi_t'(t)} +
      \bigabs{\phi_t'(t)-\phi_b'(t)}
      \lesssim \phi_a'(\abs{a-t}) + \phi_b'(\abs{b-t})
    \\
    &\lesssim \phi_a'(\abs{a-b}) +
      \phi_b'(\abs{a-b}). 
  \end{align*}
  This proves the remaining case of~\eqref{eq:shift-change-new1}. To
  prove~\eqref{eq:shift-change-new2} we estimate with
  Lemma~\ref{lem:hammer} and $(\phi_a^*)' = (\phi_a')^{-1}$
  \begin{align*}
    \bigabs{\big((\phi_a)^*\big)'(t) - \big(\phi_b)^*\big)'(t)}
    &\lesssim \big((\phi_a)^*\big)'(\abs{\phi'(a)-\phi'(b)})
    \\
    &\lesssim \big((\phi_a)^*\big)'\big(\phi_a'(\abs{b-a})\big) =
      \abs{b-a}. 
  \end{align*}
  This proves the claim. \qed
\end{proof}
With Lemma~\ref{lem:shift-change-new} we deduce the following
estimates similar to Corollary~26 and Corollary~28 of~\cite{DieK08}.
\begin{corollary}[Shift-change]
  \label{cor:shift-change-new}
  For all~$P,Q$ and all~$t \geq 0$ there holds
  \begin{align*}
    \phi_{\abs{P}}(t) &\leq (1+ c_\delta) \phi_{\abs{Q}}(t) + \delta
                        \abs{V(P) - V(Q)}^2,
                        \\
    \phi_{\abs{P}}(t) &\leq (1+ \delta) \phi_{\abs{Q}}(t) + c_\delta
                        \abs{V(P) - V(Q)}^2,
    \\
    (\phi_{\abs{P}})^*(t) &\leq (1+ c_\delta) (\phi_{\abs{Q}})^*(t) + \delta
                            \abs{V(P) - V(Q)}^2
    \\
    (\phi_{\abs{P}})^*(t) &\leq (1+ \delta) (\phi_{\abs{Q}})^*(t) + c_\delta
                            \abs{V(P) - V(Q)}^2.
  \end{align*}
\end{corollary}
\begin{proof}
  We estimate with Lemma~\ref{lem:shift-change-new}, Young's
  inequality (see Remark~\ref{rem:young}) with $\phi_{\abs{Q}}$ and Lemma~\ref{lem:hammer}
  \begin{align*}
    \phi_{\abs{P}}(t)
    &= \int_0^t \phi_{\abs{P}}'(s)\,ds
      \leq \int_0^t \phi_{\abs{P}}'(s)  + c\,
      \phi_{\abs{P}}'(\abs{P-Q})\,ds
    \\
    &=\phi_{\abs{P}}(t) +
      \phi_{\abs{P}}'(\abs{P-Q})\,t
    \\
    &=(1+c_\delta) \phi_{\abs{P}}(t) + \delta
      \phi_{\abs{P}}(\abs{P-Q})
    \\
    &\leq (1+c_\delta) \phi_{\abs{P}}(t) + \delta
      \abs{V(P)-V(Q)}^2.
  \end{align*}
  This proves the first inequality. We can exchange $\delta$
  and~$c_\delta$ within the proof to get the second inequality (see
  Remark~\ref{rem:young}).  The other estimates follow analogously
  using $(\phi_{\abs{Q}})^* = (\phi^*)_{\phi'(\abs{Q})}$.  \qed
\end{proof}

{\bf Acknowledgment.} Many thanks to Johannes Storn who did the final
numerical experiments of this article. Finally, we thank the anonymous
reviewer for the careful reading of the manuscript.

\bibliographystyle{spmpsci}

\end{document}